\documentclass[12pt,a4paper,english]{amsart}
\usepackage{mathptmx}
\usepackage[latin9]{inputenc}
\usepackage{color}
\usepackage{babel}
\usepackage{verbatim}
\usepackage{mathtools}
\usepackage{amstext}
\usepackage{amsthm}
\usepackage{float}
\usepackage{amssymb}
\usepackage{graphicx}
\usepackage{microtype}
\usepackage{enumitem}
\usepackage[unicode=true,pdfusetitle,
bookmarks=true,bookmarksnumbered=false,bookmarksopen=false,
breaklinks=false,pdfborder={0 0 1},backref=false,colorlinks=true]
{hyperref}
\hypersetup{
	pdfborderstyle=,citecolor=blue}

\makeatletter

\pdfpageheight\paperheight
\pdfpagewidth\paperwidth

\numberwithin{equation}{section}
\numberwithin{figure}{section}
\theoremstyle{plain}
\newtheorem{thm}{\protect\theoremname}[section]
\theoremstyle{plain}
\newtheorem{prop}[thm]{\protect\propositionname}
\theoremstyle{remark}
\newtheorem*{rem*}{\protect\remarkname}
\theoremstyle{definition}
\newtheorem{defn}[thm]{\protect\definitionname}
\theoremstyle{remark}
\newtheorem{rem}[thm]{\protect\remarkname}
\theoremstyle{plain}
\newtheorem{lem}[thm]{\protect\lemmaname}
\theoremstyle{plain}
\newtheorem{cor}[thm]{\protect\corollaryname}

\RequirePackage{amssymb}

\usepackage{babel}
\newcommand{\sinc}{\mathrm{sinc}}
\hypersetup{citecolor=blue}
\mathtoolsset{showonlyrefs}

\newtheorem{tm}{Theorem}[section]
\newtheorem{lemma}[tm]{Lemma}

\newcommand{\beqa}{\begin{eqnarray*}}
\newcommand{\eeqa}{\end{eqnarray*}}

\DeclareMathOperator*{\supp}{supp}
\DeclareMathOperator*{\essupp}{ess\,sup\,}

\newcommand{\field}[1]{\mathbb{#1}}
\newcommand{\bR}{\field{R}}        
\newcommand{\bN}{\field{N}}        
\newcommand{\bZ}{\field{Z}}        
\newcommand{\bC}{\field{C}}        
 \def\cH{\mathcal{H}}

 \def\cO{\mathcal{O}}

\def\rd{\bR^d}

\def\<{\left<}
\def\>{\right>}

\def\mv1{M_v^1}

\DeclareMathOperator{\trace}{tr}
\newcommand {\norm}[1]{\left\Vert #1 \right\Vert }%
\newcommand{\pwo}{PW_{[0,\Omega]}\left(A_{p}\right)}

\providecommand{\corollaryname}{Corollary}
\providecommand{\definitionname}{Definition}

\providecommand{\lemmaname}{Lemma}
\providecommand{\propositionname}{Proposition}
\providecommand{\remarkname}{Remark}
\providecommand{\theoremname}{Theorem}
\makeatother

\providecommand{\corollaryname}{Corollary}
\providecommand{\definitionname}{Definition}
\providecommand{\lemmaname}{Lemma}
\providecommand{\propositionname}{Proposition}
\providecommand{\remarkname}{Remark}
\providecommand{\theoremname}{Theorem}

\begin{document}
	\title[Variable Bandwidth]{Spectral Subspaces of Sturm-Liouville Operators and  Variable Bandwidth}
	
	\author{Mark Jason Celiz$^{1}$}
	\author{Karlheinz Gr\"ochenig$^{2}$ }
	\author{Andreas Klotz$^{2}$}
	\thanks{The authors were supported by the project P31887-N32 of the
		Austrian Science Fund (FWF). M.J.C. also acknowledges support from
		OeAD-GmbH (Austria's Agency for Education and Internationalisation).}
	
	\address{$^{\dag}$Faculty of Mathematics \\
		University of Vienna \\
		Oskar-Morgenstern-Platz 1 \\
		A-1090 Vienna, Austria}
	\address{$^{\ddag}$Institute of Mathematics\\
		University of the Philippines\\
		Diliman, 1101 Quezon City, Philippines}
	\email{mjceliz@math.upd.edu.ph}
	\email{karlheinz.groechenig@univie.ac.at}
	\email{andreas.klotz@univie.ac.at}
	\begin{abstract}
		We study spectral subspaces of the Sturm-Liouville
		operator $f \mapsto -(pf')'$ on $\bR$, where $p$ is a positive,
		piecewise constant function.  Functions in these subspaces can be
		thought of as having a local bandwidth determined by
		$1/\sqrt{p}$. Using the spectral theory of Sturm-Liouville
		operators, we make  the reproducing kernel of these spectral
		subspaces more explicit and compute it completely  in certain
		cases. As a contribution to sampling theory,  we then  prove necessary density conditions for
		sampling and interpolation in these subspaces and 
		determine the critical density that separates sets of stable
		sampling from sets of interpolation. 
	\end{abstract}
	
	\maketitle
	
	\section{Introduction}

	In~\cite{cpam} we proposed a new notion of variable bandwidth that
	connects the theory of Sturm-Liouville operators on $\bR $  with
	notions from signal processing. Roughly speaking, bandwidth is the
	largest frequency $\Omega$ in a signal (or function) $f$ on the real line,  so
	that $f$ can be represented as $f(x) = \int _{-\Omega } ^\Omega
	\hat{f}(\lambda ) e^{i\lambda x} \, d\lambda $ via the inverse Fourier
	transform. Intuitively it makes sense to speak of  a 
	maximum frequency near a time $x$  or, in other words,  of  variable
	bandwidth. However, a precise mathematical formulation of variable
	bandwidth seems difficult. The engineering literature offers several
	informal definitions~\cite{Brueller,clarkpalmerlawrence,Horiuchi1968,weioppenheim}, but so far none of them has become
	accepted. The mathematical approach  in
	~\cite{aceskafeichtinger1,aceskafeichtinger2} only yields a class of new norms on Sobolev spaces
	and lacks the fundamental features of bandwidth. The attempt of Kempf
	and Martin~\cite{kempf} uses self-adjoint extensions of a
	differential operator for an algorithmic approach to variable
	bandwidth.
	
	Our idea in~\cite{cpam} was to define variable bandwidth by means of
	the spectral subspaces of a Sturm-Liouville operator on $\bR
	$ which is closely related to the work of Pesenson and Zayed on abstract Paley-Wiener spaces~(\cite{pesensonzayed, Pesenson2001}, see also \cite{TuanZayed}). Precisely, let $p>0$ be a strictly positive, sufficiently smooth
	function on $\bR $ and consider the Sturm-Liouville operator $A_p f
	= -(pf')'$ on a suitable domain in $L^2(\bR)$ so that $A_p$ is a
	positive self-adjoint operator. The spectral theorem asserts the
	existence of a projection-valued measure $\Lambda \mapsto \chi _\Lambda
	(A_p)$. In ~\cite{cpam} we argued that the range of such an orthogonal
	projection is a good model for variable bandwidth and defined the
	Paley-Wiener space (in engineering language this is a space of bandlimited
	functions) as
	$$
	PW_{[0,\Omega]}\left(A_{p}\right)=
        \chi_{[0,\Omega]}\left(A_{p}\right)L^{2}\left(\bR\right)\,  . 
	$$
For a general Borel set $\Lambda \subset [0,\infty )$ of finite
measure, we write $PW_\Lambda (A_p) = \chi _{\Lambda } (A_p)L^2(\bR )$
for the corresponding spectral subspace.

Taking $p\equiv 1$ and $A_p = - \frac{d^2}{dx^2}$, we see that the
	Fourier transform of $A_pf$ is  
	$\widehat{A_pf}(\xi ) = \xi ^2 \hat{f}(\xi )$. Thus spectral values
	$\xi ^2 \in [0,\Omega ]$ correspond to frequencies $\xi \in [-\sqrt
	\Omega, \sqrt \Omega ]$.  Consequently, $PW_{[0,\Omega ]}(-\frac{d^2}{dx^2}) =
	\{f\in L^2(\bR ): \supp \hat{f} \subseteq [-\sqrt \Omega , \sqrt
	\Omega ]\}$ and coincides with  the classical space of bandlimited functions with
	maximum frequency $\sqrt \Omega $.\footnote{Since we deal with a second order differential
		operator,  the largest frequency $\sqrt \Omega $ corresponds to the largest spectral value
		$\Omega $. Since we use a spectral definition of variable bandwidth,
		we prefer the largest eigenvalue as the parameter in the
		definition.}
	
	In \cite{cpam} we proved several general theorems about the general
	Paley-Wiener spaces $PW_{[0,\Omega]}\left(A_{p}\right)$ to support the
	interpretation of variable bandwidth. Among these were  a general,
	almost optimal sampling theorem for the reconstruction of a function in
	$PW_{[0,\Omega]}\left(A_{p}\right)$ from its samples $\{f(x): x\in
	X\}$ on a discrete  set $X\subseteq \bR $ and a necessary density
	condition for sampling. 
	These results indicate  that a function in 
	$PW_{[0,\Omega]}\left(A_{p}\right)$ behaves like a function with  maximum
	frequency   $(\Omega / p(x))^{1/2}$  near $x$. Thus the function $p$
	defining the Sturm-Liouville operator parametrizes the local
	bandwidth.

	Despite the success of the theory and an explicit iterative
	reconstruction algorithm, an intended numerical implementation faces
	some immediate and serious difficulties. Even the most naive numerical
	treatment requires some knowledge of the reproducing kernel, that is,
	of those functions in $PW_{[0,\Omega]}\left(A_{p}\right)$ that determine
	the point evaluations $f(x) = \int _{\bR } f(y) \overline{k_x(y)} \,
	dy = \langle f, k_x \rangle $. Since $\pwo $ is a reproducing kernel
	Hilbert space~\cite[Prop.~3.3]{cpam} these functions do exist.  But what
	exactly is this reproducing kernel?
	
	Any deeper investigation of the Paley-Wiener spaces  and any attempt
	of a numerical implementation of a sampling reconstruction must use
	some knowledge about the reproducing kernel. In a sense the situation
	is similar to the theory of spaces of analytic functions (Bergman
	spaces), where some of the deepest results  hinge on delicate properties of
	the reproducing kernel (see~\cite{fefferman} for a celebrated example).  
	
	Although the spectral projections  are
	natural objects associated to a  Sturm-Liouville operator,
	surprisingly little explicit  knowledge seems to be  available. A
	general formula can be derived from spectral theory as follows. 
	Let  $\Phi(\lambda,\cdot)=(\Phi ^+(\lambda,\cdot),\Phi ^-(\lambda,\cdot))$
	be  a suitable  fundamental system of solutions of $A_{p}f=-(pf')'=\lambda f$
	on $\bR$ that depends continuously on $\lambda\in\bR$. If
        $\Lambda \subseteq [0,\infty )$ has finite Lebesgue measure,  then there
	exists a $2\times2$ positive matrix-valued Borel measure $\mu $ constructed
	from $\Phi(\lambda,\cdot)$ (see, e.g.,~\cite[Sec.~14]{weidmann2003lineare2} for details), such that the 
	reproducing kernel of  $PW_{\Lambda}(A_{p})$ is given by  
	\begin{equation}
		k_{\Lambda}(x,y)=k_x(y) =
		\int_{\Lambda}\overline{\Phi(\lambda,x)}\cdot\Phi(\lambda,y)\,d\mu(\lambda)=\langle\Phi(\cdot,y),\Phi(\cdot,x)\rangle_{L^{2}(\Lambda,\mu)}
		\, . \label{kernelformula}
	\end{equation}
	The inner product in $L^{2}([0,\Omega ] ,\mu)$ is
	\begin{align*}
		\langle F,G\rangle_{L^{2}(\bR,\mu)}=\int_{\Lambda}F(\lambda)\cdot\overline{G(\lambda)}\,d\mu(x)=\sum_{j,k=1}^{2}\int_{\Lambda}F_{j}(\lambda)\overline{G_{k}(\lambda)}\,d\mu_{jk}(\lambda)
	\end{align*}
	where the   $\mu_{jk}$'s are the component measures of $\mu$. See \cite[Sec.~XIII.5]{Dunford1988linear}
	for more details on matrix measures. 
	
	In this classical formula the kernel still depends on the knowledge of
	a fundamental solution and is far from explicit.
	
	In this paper we study the Sturm-Liouville operator $A_p f  = -(pf')'$
	for a piecewise constant, positive function $p$ and  make the reproducing kernel $k$
	as explicit as possible. Let $t_1 < t_2 < \dots < t_n$ be the position
	of the jumps of $p$ and set $t_0= -\infty $ and $t_{n+1} = +\infty
	$. Then the step function
	\begin{equation}
		\label{pwp}
		p(x) = \sum _{j=0}^n p_j \chi _{(t_j, t_{j+1}]}(x)  
	\end{equation}
	is our model to parametrize variable bandwidth. Indeed, this is the
	most natural assumption, since it can be shown that for such a 
	parametrizing function the restriction of $f\in \pwo $ to an interval
	$I_j = [t_j, t_{j+1})$ is equal to the restriction of a classical
	bandlimited function with bandwidth $(\Omega / p_j)^{1/2}$. This is
	plausible because the operator $A_p$ restricted to $L^2(I_j)$ is just
	$-p_j \frac{d^2}{dx^2}$ with eigenfunctions $e^{\pm i \lambda x}$
	corresponding to the eigenvalues $p_j \lambda ^2$. Consequently a
	spectral value $p_j \lambda ^2 \leq \Omega $ corresponds to a
	frequency $|\lambda | \leq (\Omega/p_j)^{1/2}$, cf. \cite[Prop.~3.5]{cpam}.
	
	For this natural class of  parametrizing functions we can describe the
	reproducing kernel in a more explicit fashion than in
	\eqref{kernelformula}. Let $\Phi(\lambda,\cdot)=(\Phi ^+(\lambda,\cdot),\Phi ^-(\lambda,\cdot))$
	be  a   fundamental system of solutions of $A_{p}f=-(pf')'=\lambda f$
	that depends analytically on $\lambda\in \bC \setminus (-\infty ,
	0]$. Then the reproducing kernel of $PW_\Lambda (A_p) $ is as follows. 
	
	\begin{thm}\label{kernintro}
		Let $\Lambda  \subset\bR_{0}^{+}$  and $p$ be  a
		piecewise constant, positive  function as in \eqref{pwp}. Then $PW_{\Lambda}(A_{p})$
		is a reproducing kernel Hilbert space with a  kernel of the form 
		\begin{align*}
			k_{\Lambda}(x,y)=\dfrac{1}{2\pi} \sum _{j,k = 1}^n \chi _{(t_j,
				t_{j+1}]}(x) \chi _{(t_k,
				t_{k+1}]}(y)  \int_{\Lambda^{1/2}}\dfrac{\vartheta
				_{jk} (u)}{\kappa(u)}\,du,\quad x,y\in\bR.
		\end{align*}
		All functions $\vartheta _{jk}$ and $\kappa $ depend only on the
		parameters 
		$\{p_j: j=1,\dots , n\}$ and the jumps $\{t_j\}$ of $p$. Furthermore,
		each  function  $\vartheta _{jk}$ and $\kappa$ are  almost periodic polynomials 
		with real coefficients. 
	\end{thm}
	
	For  the spectrum $\Lambda = [0,\Omega ]$ and $n=2$, i.e., $p$ has two jumps and partitions $\bR $ into three
	intervals of constant local bandwidth, $\kappa$ can be explicitly
	computed to be (set $q_k = p_k^{-1/2}$, $k = 0,1$)
	\begin{equation}
		\label{eq:cc1}
		\kappa (u) =
		\tfrac{1}{16q_{0}^{2}}\left[\left(1+\tfrac{q_{0}}{q_{1}}\right)^{2}\left(1+\tfrac{q_{1}}{q_{2}}\right)^{2}+\left(1-\tfrac{q_{0}}{q_{1}}\right)^{2}\left(1-\tfrac{q_{1}}{q_{2}}\right)^{2}\right]
		+
		\tfrac{1}{8q_{0}^{2}}\left(1-\tfrac{q_{0}^{2}}{q_{1}^{2}}\right)\left(1-\tfrac{q_{1}^{2}}{q_{2}^{2}}\right)
		\,  \cos 2q_{1}T u \, .
	\end{equation}
	The kernel  $k$ is then obtained by evaluating  parameter integrals of
	the form 
	\begin{align*}
		J(s)=\int_{\Lambda^{1/2}}\frac{e^{isu}}{C+ K \cos \zeta u}\,du,\quad s\in\bR.
	\end{align*}
	This integral seems to be a new type of special function. The kernel
	can then be written explicitly by means of this function $J$. This
	will be done in Theorem~\ref{ker3pardami}.  
	
	In the second part of the paper  we derive necessary density conditions for sampling and
	interpolation in $PW_{\Lambda}(A_{p})$ in the style of Landau
	\cite{landau1967} for the case of a piecewise constant
	parametrizing  function $p$. As
	in \cite[Sec.~6]{cpam}, we consider the positive measure 
	\begin{align}
		\mu_{p}([a,b])=\int_a^b\dfrac{dx}{\sqrt{p(x)}}\label{mup} \, ,
	\end{align}
	which is determined
	by the parametrizing function $p$. We also recall the following notions.
	We say $X$ is a \emph{set of stable sampling} for $PW_{\Lambda}(A_{p})$
	if there exist $C_{1},C_{2}>0$ such that 
	\begin{align*}
		C_{1}\norm{f}_{2}^{2}\leq\sum_{x\in X}|f(x)|^{2}\leq
		C_{2}\norm{f}_{2}^{2} \quad \text{ for all } f\in
		PW_{\Lambda}(A_{p})\, .
	\end{align*}
	Furthermore, $X$ is a
	\emph{set of interpolation} for $PW_{\Lambda}(A_{p})$ if for every
	$c\in\ell^{2}(X)$ there exists $f\in PW_{\Lambda}(A_{p})$ such that
	$f(x)=c_{x}$ for all $x\in X$. Using a modified form of the upper
	and lower Beurling densities via $\mu_{p}$, we derive the following
	version of the density theorem. 
	\begin{thm}[Density theorem in $PW_{\Lambda}(A_{p})$]
		\label{densintro}
		Let 
		$p$ be a piecewise constant function and $\Omega >0$. 
		\begin{enumerate}
			\item If $X$ is a set of stable sampling for $PW_{[0,\Omega ]}(A_{p})$, then
			\begin{align*}
				D_{p}^{-}(X)=\liminf_{r\to\infty}\inf_{x\in\bR}\dfrac{\#(X\cap
					B_{r}(x))}{\mu_{p}(B_{r}(x))}\geq\frac{\Omega^{1/2}}{\pi}.
			\end{align*}
			\item If $X$ is a set of interpolation for $PW_{[0,\Omega]}(A_{p})$, then
			\begin{align*}
				D_{p}^{+}(X)=\limsup_{r\to\infty}\sup_{x\in\bR}\dfrac{\#(X\cap
					B_{r}(x))}{\mu_{p}(B_{r}(x))}\leq\frac{\Omega ^{1/2}}{\pi}.
			\end{align*}
		\end{enumerate}
	\end{thm}
	Although the density theorem does not make any reference to
	the reproducing kernel,  Theorem~\ref{kernintro} is substantial for
	its proof. In the proof of Theorem~\ref{densintro} we will apply a
	general density theorem for sampling and interpolation in 
	reproducing kernel Hilbert spaces \cite{fuehr}. This statement
	identifies an averaged trace of the reproducing kernel as the
	critical density. Based on the semi-explicit formula of
	Theorem~\ref{kernintro} we will show that the averaged trace equals
	$\sqrt{\Omega }/\pi $ (see Theorem~\ref{limmup}).
	
	Theorem~\ref{densintro} holds for much more general spectra $\Lambda
	$. In fact, if $\Lambda \subset \bR ^+$ has finite Lebesgue measure, then
	 the necessary condition for sampling is
	$$
	D_p^-(X) \geq \tfrac{1}{\pi } \big| \{ x \in \bR : x^2 \in \Lambda \}
	\big| \, ,
	$$
	where $|I|$ denotes the Lebesgue measure of a set $I\subseteq \bR $.
	
	Finally, the  formulation of the necessary densities  is identical to  the main
	result  in \cite[Thm.~6.2, 6.3]{cpam}.  However, the assumptions on
	the parametrizing function $p$ are radically different. In~\cite{cpam}
	we assumed $p$ to be smooth, at least $C^2$, and then transformed the
	Sturm-Liouville operator into an equivalent Schr\"odinger operator and
	applied scattering theory. In the multivariate version in~\cite{apde}
	we even assumed $p\in C^\infty $ and used the regularity theory of
	elliptic partial differential equations to derive information about
	the reproducing kernel. For the case of a singular function $p$ all
	these tools fail dramatically; our main tool is the simple form of the
	operator $A_p$ and the ensuing more explicit knowledge about the
	reproducing kernel.

	\subsection{Organization}
	
	The paper is organized as follows: in Section \ref{sec:prel} we recall
	a few concepts and results from Sturm-Liouville theory. In Section
	\ref{sec:PW} we 
	introduce rigorously the Paley-Wiener space $PW_{\Lambda}(A_{p})$. 
	We will determine  a suitable  fundamental system of solutions and  the spectral
	measure $\mu$ for $A_p$.  In Section
	\ref{sec:ex} we compute  the reproducing kernel   of
	$\pwo $ for a piecewise constant parametrizing function $p$. In
	Section \ref{sec:dens}, 
	we derive necessary density conditions
	for sampling and interpolation in $PW_{\Lambda}(A_{p})$ with  piecewise
	constant $p$. 

        \vspace{3mm}

	\textbf{Acknowledgment}: We would like to thank our colleagues Gerald Teschl
	and Jussi Behrndt for discussion and advice about the spectral theory
	of Sturm-Liouville operators.

	\section{Spectral Theory of Sturm-Liouville Operators}
	
	\label{sec:prel} This section, adapted from \cite{cpam}, is a brief
	review of the relevant spectral theory of singular Sturm-Liouville
	differential expressions on $\bR$ in divergence form. Given a positive
	function $p$ on $\bR $, we define the
	differential expression $\tau _p$ by $\tau_{p}f=-\left(pf'\right)'$.
	\begin{prop} \label{correctdef}
		If $p$ is a piecewise constant function as defined in \eqref{pwp}, 
		then $\tau _p$ with domain
		\begin{align*}
			\mathcal{D}(A_{p}) & =\{f\in L^{2}(\bR):f,pf'\in AC_{loc}(\bR)\textrm{ and }\tau_{p}f\in L^{2}(\bR)\},\\
			A_{p}f & =\tau_{p}f,f\in\mathcal{D}(A_{p})
		\end{align*}
		defines the self-adjoint operator $A_{p}$ on $L^{2}\left(\bR\right)$.
		The spectrum of $A_{p}$ is purely absolutely continuous and consists
		of the positive semiaxis: $\sigma\left(A_{p}\right)=\sigma_{ac}\left(A_{p}\right)=[0,\infty)$.
	\end{prop}
	
	\begin{proof}
		The assertions concerning self-adjointness are standard, see, e.g.,
		\cite{teschl2014mathematical,weidmann2003lineare2,Dunford1988linear}.
		The result on the purely absolutely continuous spectrum can be inferred
		form \cite{Schmied_2008}. %
	\end{proof}
	We also have the following spectral representation of $A_{p}$ (see
	\cite{cpam}).%

	\begin{prop}
		\label{weidspecthm} If $\Phi(\lambda,x)=\big(\phi_{1}(\lambda,x),\phi_{2}(\lambda,x)\big)$
		is a fundamental system of solutions of $(\tau-\lambda)\phi=0$ that
		depends continuously on $\lambda$, then there exists a $2\times2$
		matrix measure $\rho$, such that the operator 
		\begin{equation}
			\mathcal{F}_{A_{p}}:L^{2}(\bR)\to L^{2}(\bR,d\rho);\quad\mathcal{F}_{A_{p}}f(\lambda)=\int_{\bR}f(x)\,\overline{\Phi(\lambda,x)}dx\label{eq:c21}
		\end{equation}
		is unitary and diagonalizes $A$, i.e.,
		$$\mathcal{F}_{A_{p}}A\mathcal{F}_{A_{p}}^{-1}G(\lambda)=\lambda G(\lambda)$$
		for all $G\in L^{2}(\bR,d\rho)$. The inverse has the form 
		\[
		{\mathcal{F}_{A_{p}}^{-1}}G(x)=\int_{\bR}G(\lambda)\cdot{\Phi(\lambda,x)}\,d\rho(\lambda)\,
		.
		\]
		for $G\in L^{2}(\bR,d\rho)$.
		
		If $g$ is a bounded Borel function on $\bR$, then for every $f\in L^{2}(\bR)$,
		\begin{equation}
			g(A)f(x)=\int_{\bR_{0}^{+}}g(\lambda)\mathcal{F}_{A_{p}}f(\lambda)\cdot{\Phi(\lambda,x)}\,d\rho(\lambda)\,.\label{eq:specrep}
		\end{equation}
		All integrals $\int_{\bR}$ have to be understood as $\lim_{\substack{a\to-\infty\\
				b\to\infty
			}
		}\int_{a}^{b}$ with convergence in $L^{2}$. 
	\end{prop}
	
	$\mathcal{F}_{A_{p}}$ is called the \emph{spectral transform} (or
	also a spectral representation of $A$).
	\begin{rem*}
		It is always possible to choose a fundamental system of solutions
		$\Phi(\lambda,\cdot)$ that depends continuously (even analytically)
		on $\lambda$~\cite{weidmann2003lineare2}. The spectral measure
		can (and will be) then be constructed explicitly from the knowledge
		of such a set of fundamental solutions $(A-z)\phi=0$, see \cite{Dunford1988linear,teschl2014mathematical,Weidmann1987spectral,weidmann2003lineare2}
		.
	\end{rem*}

	In particular, the \textit{spectral projection} $\chi_{\Lambda}(A_{p}):L^{2}(\bR)\to L^{2}(\bR)$
	corresponding to a Borel set $\Lambda\subseteq\bR_{0}^{+}$ is given
	by 
	
	\begin{align}
		\chi_{\Lambda}(A_{p})f=\int_{\Lambda}\mathcal{F}_{A_{p}}f(\lambda)\cdot\Phi(\lambda,\cdot)\,d\mu(\lambda)=\mathcal{F}_{A_{p}}^{-1}(\chi_{\Lambda}\mathcal{F}_{A_{p}}f)\label{specproj}
	\end{align}
	for all $f\in L^{2}(\bR)$. With this projection operator we define
	the main object for our approach to variable bandwidth. 
	\begin{defn}
		Let $A_{p}$ as defined in Proposition~\ref{correctdef}
		{} and $\Lambda\subset\bR_{0}^{+}$ be of finite measure. The Paley-Wiener
		space, denoted $PW_{\Lambda}(A_{p})$,
		is the range of $\chi_{\Lambda}(A_{p})$, i.e., 
		\[
		PW_{\Lambda}(A_{p})=\chi_{\Lambda}(A_{p})(L^{2}(\bR)).
		\]
		Equivalently, a function $f\in L^{2}(\bR)$ belongs to $PW_{\Lambda}(A_{p})$
		if $f=\chi_{\Lambda}(A_{p})f$.
	\end{defn}
	
	Some elementary properties of functions of variable bandwidth can
	be found in \cite[Sec.~3]{cpam}. As argued in \cite{cpam}, this
	definition is one possibility to give meaning to the notion of
	variable bandwidth. 
	
	\section{Paley Wiener Space with Piecewise Constant Parametrization }
	
	\label{sec:PW} We now present the fundamental results on $PW_{\Lambda}(A_{p})$
	for $p$ piecewise constant as in \eqref{pwp}.
	We will derive a formula for the spectral measure and 
	discuss a strategy how to compute the reproducing kernel.%

	\subsection{Fundamental solutions of $(\tau_{p}-\lambda)f=0$ \label{subsec:Fundamental-solutions-of}}
	
	In order to derive the spectral representation $\mathcal{F}_{A_{p}}$,
	we choose a fundamental system $\Phi(z,\cdot)=(\Phi^{+}(z,\cdot),\Phi^{-}(z,\cdot))$
	of $(\tau_{p}-z)f=0$ of (classical) solutions that depends analytically on $z\in\bC\setminus\bR$
	and such that one solution lies right and the other lies left in $L^2(\bR)$ (a function $f$ \emph{lies right} in $L^{2}(\bR)$
		if $f\in L^{2}(c,\infty)$ for some $c\in\bR$. Similarly, $f$ \emph{lies
			left} in $L^2(\bR)$ if $f\in L^{2}(-\infty, c)$ for some $c \in \bR$). These additional integrability conditions are in preparation
	for the derivation of the spectral measure.
	
	The strategy is as follows. Since $p$ is constant 
	$p(x)=p_{k}$ on the interval  $I_{k}= [t_k,t_{k+1})$, the eigenvalue equation
	$A_pf=zf$ becomes $(-p_{k}D^{2}-z)f=0$ on $ I_{k}$, and therefore
	every eigenfunction of $A_p$ restricted to $I_k$ possesses an
	elementary solution by exponentials. To obtain an eigenfunction for
	$A_p$ on all of $\bR $, we have to glue together these local
	solutions, which is a construction similar to that of
	splines. In the following we use once and for all the notation
	\begin{equation}
		\label{eq:cc2}
		q_k=\frac{1}{\sqrt{p_k}} \, ,
	\end{equation}
	as this is the precise local frequency of the eigenfunctions. 
	We also use the principal square root of $z$ defined as follows: if $z=re^{i\theta} \in \bC\setminus(-\infty,0]$ with $r>0$ and $-\pi < \theta < \pi,$ then $\sqrt{z} = \sqrt{r}e^{i\theta/2}$. In particular, $\textrm{Im}z$ and $\textrm{Im}\sqrt{z}$ have the same sign.
	\begin{thm}
		\label{lemsol} Let $p$ be a piecewise constant function. Then there
		exist analytic functions $a_{l}^{+},b_{l}^{+},a_{j}^{-},b_{j}^{-}$
		from $\bC\setminus(-\infty,0]$ to $\bC$ such that, 
		for every $z\in\bC\setminus(-\infty,0]$,  the functions
		\begin{align}
			\Phi^{+}(z,x) & =\begin{cases}
				a_{l}^{+}(z)e^{iq_{l}\sqrt{z}x}+b_{l}^{+}(z)e^{-iq_{l}\sqrt{z}x}, & x\in I_{l},0\leq l\leq n-1,\\
				e^{iq_{n}\sqrt{z}x}, & x\in I_{n},
			\end{cases}\label{pw1}\\
			\Phi^{-}(z,x) & =\begin{cases}
				e^{-iq_{0}\sqrt{z}x}, & x\in I_{0},\\
				a_{j}^{-}(z)e^{iq_{j}\sqrt{z}x}+b_{j}^{-}(z)e^{-iq_{j}\sqrt{z}x}, & x\in I_{j},1\leq j\leq n
			\end{cases}\label{pw2}
		\end{align}
		are solutions of $(\tau_{p}-z)f=0$ that are analytic in $z$ on $\bC\setminus(-\infty,0]$.

		If $\textrm{Im}z>0$, then $\Phi(z,\cdot) = \big(\Phi ^+(z,\cdot ) , \phi
		^-(z,\cdot )\big)$ is a fundamental system
		of $(\tau_{p}-z)f=0$ such that $\Phi^{+}(z,\cdot)$ lies right and
		$\Phi^{-}(z,\cdot)$ lies left in $L^{2}(\bR)$, and 
		
		Similarly, if $\textrm{Im}z<0$, then $\overline{\Phi(\overline{z},\cdot)}$
		is a fundamental system of $(\tau_{p}-z)f=0$ such that $\overline{\Phi^{+}(\overline{z},\cdot)}$
		lies right and $\overline{\Phi^{-}(\overline{z},\cdot)}$ lies left
		in $L^{2}(\bR)$. 
		
		Setting
		\begin{equation}
			\label{eq:cc3}
			L_{k}(z)=\dfrac{1}{2}\begin{bmatrix}\left(1+\frac{q_{k}}{q_{k-1}}\right)e^{it_{k}(q_{k-1}-q_{k})\sqrt{z}} & \left(1-\frac{q_{k}}{q_{k-1}}\right)e^{-it_{k}(q_{k-1}+q_{k})\sqrt{z}}\\
				\left(1-\frac{q_{k}}{q_{k-1}}\right)e^{it_{k}(q_{k-1}+q_{k})\sqrt{z}} & \left(1+\frac{q_{k}}{q_{k-1}}\right)e^{-it_{k}(q_{k-1}-q_{k})\sqrt{z}}
			\end{bmatrix}\, ,
		\end{equation}
		the coefficients $a^{\pm}, b^{\pm}$ satisfy the recursion relations 
		\[
		\begin{bmatrix}a_{k}^{\pm}(z)\\
			b_{k}^{\pm}(z)
		\end{bmatrix}=L_{k}(z)\begin{bmatrix}a_{k-1}^{\pm}(z)\\
			b_{k-1}^{\pm}(z)
		\end{bmatrix}\, .
		\]
		As a consequence the coefficients $a^{\pm}, b^{\pm }$ are almost
		periodic trigonometric polynomials of the variable $\sqrt{z}$. 
	\end{thm}
	
	\begin{proof}
		Fix $z\in\bC\setminus(-\infty,0]$. On the interval $I_k = [t_k,
		t_{k+1})$ the parametrizing function $p$ is constants and we need to
		solve
		$$
		(A_p - z)f(x) = (- p_k \frac{d^2f}{dx^2} - zf)(x) = 0 \qquad \for x\in
		I_k \, .
		$$
		It is clear that  these  local solutions $\varphi_{k}(z,\cdot)$
		take the form 
		\begin{align*}
			\varphi_{k}(z,x)=a_{k}(z)e^{iq_{k}\sqrt{z}x}+b_{k}(z)e^{-iq_{k}\sqrt{z}x},\quad x\in I_{k}
		\end{align*}
		for some choice of constants $a_{k}(z),b_{k}(z)\in\bC,0\leq k\leq n.$
		We need to choose these constants in such a way that the function 
		\[
		\varphi(z,x)= \sum _{k=0}^n \varphi_{k}(z,x)\chi _{I_k}(x) 
		\]
		satisfies
		$\varphi(z,\cdot),p\varphi'(z,\cdot)\in AC_{loc}(\bR)$. These
		conditions 
		imply that for each $k$, 
		\begin{align*}
			\begin{bmatrix}\varphi_{k-1}(z,t_{k})\\
				\frac{1}{q_{k-1}^{2}}\varphi_{k-1}'(z,t_{k})
			\end{bmatrix}=\begin{bmatrix}\varphi(z,t_{k})\\
				p(t_{k})\varphi'(z,t_{k})
			\end{bmatrix} & =\lim_{x\downarrow t_{k}}\begin{bmatrix}\varphi(z,x)\\
				p(x)\varphi'(z,x)
			\end{bmatrix}=\begin{bmatrix}\varphi_{k}(z,t_{k})\\
				\frac{1}{q_{k}^{2}}\varphi_{k}'(z,t_{k})
			\end{bmatrix} \, .
		\end{align*}
		Upon substituting the local solutions, this yields
		\begin{align*}
			\begin{bmatrix}e^{iq_{k-1}\sqrt{z}t_{k}} & e^{iq_{k-1}\sqrt{z}t_{k}}\\
				i\frac{\sqrt{z}}{q_{k-1}}e^{iq_{k-1}\sqrt{z}t_{k}} & -i\frac{\sqrt{z}}{q_{k-1}}e^{iq_{k-1}\sqrt{z}t_{k}}
			\end{bmatrix}\begin{bmatrix}a_{k-1}(z)\\
				b_{k-1}(z)
			\end{bmatrix} & =\begin{bmatrix}e^{iq_{k}\sqrt{z}t_{k}} & e^{iq_{k}\sqrt{z}t_{k}}\\
				i\frac{\sqrt{z}}{q_{k}}e^{iq_{k}\sqrt{z}t_{k}} & -i\frac{\sqrt{z}}{q_{k}}e^{iq_{k}\sqrt{z}t_{k}}
			\end{bmatrix}\begin{bmatrix}a_{k}(z)\\
				b_{k}(z)
			\end{bmatrix}.
		\end{align*}
		With the notation for the matrices $L_k(z)$ from \eqref{eq:cc3},
		$\varphi(z,\cdot)$ is a solution of
		$(\tau_{p}-z)f=0$ 
		if and only if the relations 
		\begin{align*}
			\begin{bmatrix}a_{k}(z)\\
				b_{k}(z)
			\end{bmatrix}=L_{k}(z)\begin{bmatrix}a_{k-1}(z)\\
				b_{k-1}(z)
			\end{bmatrix}
		\end{align*}
		hold for all $k$. In particular, taking $a_{0}^{-}(z)=0$, $b_{0}^{-}(z)=1$
		and recursively generating the remaining coefficients via 
		\begin{align}
			\begin{bmatrix}a_{j}^{-}(z)\\
				b_{j}^{-}(z)
			\end{bmatrix}=L_{j}(z)\begin{bmatrix}a_{j-1}^{-}(z)\\
				b_{j-1}^{-}(z)
			\end{bmatrix},\quad1\leq j\leq n,\label{aklk}
		\end{align}
		we generate the solution $\Phi^{-}(z,\cdot)$ of $(\tau_{p}-z)f=0,z\in\bC\setminus(-\infty,0]$
		as claimed in \eqref{pw2}. Analogously, we generate $\Phi^{+}$ in \eqref{pw1}
		using 
		\begin{align}
			\qquad R_{k}(z)=L_{k}^{-1}(z)=\dfrac{1}{2}\begin{bmatrix}\left(1+\frac{q_{k-1}}{q_{k}}\right)e^{-it_{k}(q_{k-1}-q_{k})\sqrt{z}} & \left(1-\frac{q_{k-1}}{q_{k}}\right)e^{-it_{k}(q_{k-1}+q_{k})\sqrt{z}}\\
				\left(1-\frac{q_{k-1}}{q_{k}}\right)e^{it_{k}(q_{k-1}+q_{k})\sqrt{z}} & \left(1+\frac{q_{k-1}}{q_{k}}\right)e^{it_{k}(q_{k-1}-q_{k})\sqrt{z}}
			\end{bmatrix},\label{Rj}
		\end{align}
		the initial values
		\begin{equation}
			\label{eq:cc4}
			a_{n}^{+}(z)=1,b_{n}^{+}(z)=0
		\end{equation}
		and the recursions
		\begin{align}
			\begin{bmatrix}a_{l}^{+}(z)\\
				b_{l}^{+}(z)
			\end{bmatrix}=R_{l+1}(z)\begin{bmatrix}a_{l+1}^{+}(z)\\
				b_{l+1}^{+}(z)
			\end{bmatrix},\quad0\leq l\leq n-1 \, .\label{albl}
		\end{align}
		The analyticity of $a_{k}^{\pm}$, $b_{k}^{\pm},$ and consequently
		of $\Phi^{+}$ and $\Phi^{-}$ in $z$ follows from the analyticity
		of $z\mapsto e^{\pm iq_{k}\sqrt{z}x}$ in $\bC\setminus(-\infty,0]$
		for all $x\in\bR$. Finally, if $\textrm{Im}z>0$ then $\textrm{Im}\sqrt{z}>0$.
		Moreover, 
		\begin{align*}
			|e^{\pm iq\sqrt{z}x}|=e^{\mp q\textrm{Im}\sqrt{z}x},\quad x\in\bR,q>0,z\in\bC\setminus(-\infty,0].
		\end{align*}
		Thus, $\Phi^{-}(z,x)=e^{-iq_{0}\sqrt{z}x}\in L^{2}(I_{0})$ and
		$\Phi^{+}(z,x)=e^{iq_{n}\sqrt{z}x}\in L^{2}(I_{n})$. This means that  $\Phi^{-}(z,\cdot)$ lies left and $\Phi^{+}(z,\cdot)$ lies
		right in $L^{2}(\bR)$. The case $\textrm{Im}z<0$ for $\overline{\Phi^{+}(\overline{z},\cdot)}$
		and $\overline{\Phi^{-}(\overline{z},\cdot)}$ is proved analogously. 
	\end{proof}
	We call the analytic functions $a_{k}^{\pm}$ and $b_{k}^{\pm}$ the
	\emph{connection coefficients} of $\Phi^{+}$ and $\Phi^{-}$. From
	the iterative computations in \eqref{aklk} and \eqref{albl} we see
	that 
	\begin{align}
		\begin{bmatrix}a_{l}^{+}(z)\\
			b_{l}^{+}(z)
		\end{bmatrix} & =R_{l+1}(z)R_{l+2}(z)\cdots R_{n}(z)\begin{bmatrix}1\\
			0
		\end{bmatrix},\quad0\leq l\leq n-1\label{reccoefrj}\\
		\begin{bmatrix}a_{j}^{-}(z)\\
			b_{j}^{-}(z)
		\end{bmatrix} & =L_{j}(z)L_{j-1}(z)\cdots L_{1}(z)\begin{bmatrix}0\\
			1
		\end{bmatrix},\quad1\leq j\leq n.\label{reccoeflk}
	\end{align}
	These coefficients are used to continuously glue together the local
	solutions of $(\tau_{p}-z)f=0$ on each $I_{k}$ to form two global
	solutions $\Phi^{+}(z,\cdot)$, $\Phi^{-}(z,\cdot)$ that lie right
	and lie left in $L^{2}(\bR)$, respectively. Moreover, the above formulas
	show that the connection coefficients are almost periodic polynomials
	in $\sqrt{z}$ with real coefficients. 
	\begin{rem}
		\label{surmk} Hidden in the recursion relations \eqref{reccoefrj} is a group
		structure. To see this, define for $1\leq k\leq n$ the quantities 
		\[
		\gamma_{k}=(\tfrac{q_{k-1}}{q_{k}})^{1/2},\quad\eta_{k}=t_{k}(q_{k-1}-q_{k})\quad\textrm{ and }\quad\theta_{k}=t_{k}(q_{k-1}+q_{k}).
		\]
		We can then express $L_{k}(z)$ as 
		\begin{align*}
			L_{k}(z) & =\gamma_{k}^{-1}\begin{bmatrix}\frac{\gamma_{k}+\gamma_{k}^{-1}}{2}e^{i\eta_{k}\sqrt{z}} & \frac{\gamma_{k}-\gamma_{k}^{-1}}{2}e^{-i\theta_{k}\sqrt{z}}\\
				\frac{\gamma_{k}-\gamma_{k}^{-1}}{2}e^{i\theta_{k}\sqrt{z}} & \frac{\gamma_{k}+\gamma_{k}^{-1}}{2}e^{-i\eta_{k}\sqrt{z}}
			\end{bmatrix},\quad\det L_{k}(z)=\gamma_{k}^{-2}
		\end{align*}
		and \eqref{Rj} as 
		\begin{align*}
			R_{k}(z) & =\gamma_{k}\begin{bmatrix}\frac{\gamma_{k}+\gamma_{k}^{-1}}{2}e^{-i\eta_{k}\sqrt{z}} & -\frac{\gamma_{k}-\gamma_{k}^{-1}}{2}e^{-i\theta_{k}\sqrt{z}}\\
				-\frac{\gamma_{k}-\gamma_{k}^{-1}}{2}e^{i\theta_{k}\sqrt{z}} & \frac{\gamma_{k}+\gamma_{k}^{-1}}{2}e^{i\eta_{k}\sqrt{z}}
			\end{bmatrix},\quad\det R_{k}(z)=\gamma_{k}^{2}.
		\end{align*}
		For $z\in(0,\infty)$, the matrices $\gamma_{k}L_{k}(z)$ and $\gamma_{k}^{-1}R_{k}(z)$
		are elements of the Lie group 
		\begin{align*}
			SU(1,1)=\left\{ \begin{bmatrix}\overline{a} & b\\
				\overline{b} & a 
			\end{bmatrix}:a,b\in\bC\textrm{ and }|a|^{2}-|b|^{2}=1\right\} .
		\end{align*}
		Hence,
		\begin{align*}
			\gamma_{j}L_{j}(z)\cdot\ldots\cdot\gamma_{1}L_{1}(z)=\left(\prod_{k=1}^{j}\gamma_{k}\right)\begin{bmatrix}\overline{b_{j}^{-}(z)} & a_{j}^{-}(z)\\
				\overline{a_{j}^{-}(z)} & b_{j}^{-}(z)
			\end{bmatrix}
		\end{align*}
		and therefore 
		\begin{align}
			L_{j}(z)\cdot\ldots\cdot L_{1}(z)=\begin{bmatrix}\overline{b_{j}^{-}(z)} & a_{j}^{-}(z)\\
				\overline{a_{j}^{-}(z)} & b_{j}^{-}(z)
			\end{bmatrix},\quad1\leq j\leq n.\label{ljprod}
		\end{align}
		By the same token, it can be shown that 
		\begin{align}
			R_{l+1}(z)\cdot\ldots\cdot R_{n}(z)=\begin{bmatrix}a_{l}^{+}(z) & \overline{b_{l}^{+}(z)}\\
				b_{l}^{+}(z) & \overline{a_{l}^{+}(z)}
			\end{bmatrix},\quad0\leq l\leq n-1.\label{rlprod}
		\end{align}
	\end{rem}
	
	\begin{lem}
		Let $\Phi=(\Phi^{+},\Phi^{-})$ be defined as in Theorem \ref{lemsol}.
		Then $\Phi$ is uniformly bounded on $\bR^{+}\times\bR$. \label{unifbdd} 
	\end{lem}
	
	\begin{proof}
		By definition of $\Phi$, we see that for $0\leq k\leq n$ and $(\lambda,x)\in\bR^{+}\times I_{k}$,
		\begin{align*}
			|\Phi^{\pm}(\lambda,x)|\leq|a_{k}^{\pm}(\lambda)|+|b_{k}^{\pm}(\lambda)|=\norm{\begin{bmatrix}a_{k}^{+}(\lambda) & b_{k}^{+}(\lambda)\end{bmatrix}^{T}}_{1}.
		\end{align*}
		Let the matrix norm subordinate to $\lVert\cdot\rVert_{1}$ be denoted
		by the same symbol. From \eqref{reccoefrj} %
		we have that for $(\lambda,x)\in\bR^{+}\times\bR$, 
		\begin{align*}
			|\Phi^{+}(\lambda,x)| & \leq\max_{0\leq k\leq n}\norm{\begin{bmatrix}a_{k}^{+}(\lambda) & b_{k}^{+}(\lambda)\end{bmatrix}^{T}}_{1}\leq\max\left\{ 1,\max_{1\leq k\leq n}\norm{R_{k}R_{k-1}\cdots R_{n}}_{1}\right\} \,.
		\end{align*}
		%

		From %
		{} \eqref{Rj}, we obtain for $1\leq k\leq n$ the estimates 
		\begin{align*}
			\norm{R_{k}}_{1} & =\tfrac{1}{2}\left(\left|1+\tfrac{q_{k-1}}{q_{k}}\right|+\left|1-\tfrac{q_{k-1}}{q_{k}}\right|\right)\leq1+\tfrac{q_{k-1}}{q_{k}}
		\end{align*}
		The assertion now follows from the submultiplicativity of $\norm{\cdot}_{1}$
		. The estimate for $\Phi^{-}(\lambda,x)$ follows the same lines.
	\end{proof}
	Next, we derive several fundamental  identities
	for the connection coefficients.  These  expressions will follow from
	properties of a  (modified) Wronskian determinant. 
	\begin{lem}
		\label{wronskthm} Let $n\in\bN$ and $a_{k}^{\pm},b_{k}^{\pm},0\leq
		k\leq n$, 
		be the connection coefficients of $\Phi=(\Phi^{+},\Phi^{-})$. Then
		the following identities hold. 
		
		(i) If $z\in\bC\setminus(-\infty,0]$, then, for all $1\leq k\leq n-1$
		\begin{align}
			\frac{a_{0}^{+}(z)}{q_{0}} & =\frac{1}{q_{k}}\left(a_{k}^{+}(z)b_{k}^{-}(z)-a_{k}^{-}(z)b_{k}^{+}(z)\right)=\frac{b_{n}^{-}(z)}{q_{n}}\label{wron1}
		\end{align}
		
		(ii) If $\lambda\in(0,\infty)$, then for all $0\leq k\leq n$, 
		\begin{align}
			q_{0}(|b_{k}^{-}(\lambda)|^{2}-|a_{k}^{-}(\lambda)|^{2}) & =q_{n}(|a_{k}^{+}(\lambda)|^{2}-|b_{k}^{+}(\lambda)|^{2})=q_{k}\,.\label{id1}
		\end{align}
		Consequently, 
		\begin{equation}
			\frac{|a_{k}^{+}(\lambda)|^{2}}{q_{0}}+\frac{|a_{k}^{-}(\lambda)|^{2}}{q_{n}}=\frac{|b_{k}^{+}(\lambda)|^{2}}{q_{0}}+\frac{|b_{k}^{-}(\lambda)|^{2}}{q_{n}}.\label{hidentity}
		\end{equation}
		
		(iii) If $\lambda\in(0,\infty)$, then  for all $1\leq k\leq n-1$
		\begin{align}
			\frac{b_{0}^{+}(\lambda)}{q_{0}} & =\frac{1}{q_{k}}\left(b_{k}^{+}(\lambda)\overline{b_{k}^{-}(\lambda)}-a_{k}^{+}(\lambda)\overline{a_{k}^{-}(\lambda)}\right)=-\frac{\overline{a_{n}^{-}(\lambda)}}{q_{n}}\,.\label{wron3}
		\end{align}
	\end{lem}

	\begin{proof}
		Let
		\begin{align*}
			W_{p}(u,v)=\begin{vmatrix}u & v\\
				pu' & pv'
			\end{vmatrix}=u(pv')-v(pu')
		\end{align*}
		be the modified Wronskian determinant of a pair of solutions $u,v$
		of $(A_p-z)f=0$ and recall the important fact that  the Wronskian of solutions
		of the differential equation is a constant and independent of the
		variable $x$ (see, e.g., \cite{teschl2014mathematical,weidmann2003lineare2}).
		We apply this fact  to the components of $\Phi ^+$ and $\Phi ^-$ on each  interval $I_{k}$.
		
		(i) Let $z\in\bC\setminus(-\infty,0]$. Then for $x\in I_{k},0\leq k\leq n$
		and  
		\begin{align*}
			\Phi^{\pm}(z,x)=a_{k}^{\pm}(z)e^{iq_{k}\sqrt{z}x}+b_{k}^{\pm}(z)e^{-iq_{k}\sqrt{z}x},
		\end{align*}
		we compute directly that 
		\begin{align}
			\lefteqn{W_{p}\left(\Phi^{+}(z,x),\Phi^{-}(z,x)\right)=} \\ &
			=a_{k}^{+}(z)b_{k}^{-}(z)W\left(e^{iq_{k}\sqrt{z}x},e^{-iq_{k}\sqrt{z}x}\right) +b_{k}^{+}(z)a_{k}^{-}(z)W\left(e^{-iq_{k}\sqrt{z}x},e^{iq_{k}\sqrt{z}x}\right)\nonumber \\
			& =-\tfrac{2i\sqrt{z}}{q_{k}}\left(a_{k}^{+}(z)b_{k}^{-}(z)-a_{k}^{-}(z)b_{k}^{+}(z)\right).\label{finint}
		\end{align}
		On the unbounded intervals $I_{0}$ and $I_{n}$, \eqref{finint}
		reduces to 
		\begin{align}
			W_{p}\left(\Phi^{+}(z,x),\Phi^{-}(z,x)\right) & =-\tfrac{2i\sqrt{z}}{q_{0}}a_{0}^{+}(z)\quad(k=0:a_{0}^{-}(z)=0,b_{0}^{-}(z)=1)\label{ioint}\\
			W_{p}\left(\Phi^{+}(z,x),\Phi^{-}(z,x)\right) & =-\tfrac{2i\sqrt{z}}{q_{n}}b_{n}^{-}(z)\quad(k=n:a_{n}^{+}(z)=1,b_{n}^{+}(z)=0).\label{inint}
		\end{align}
		The equality of all three expressions yields \eqref{wron1} for all
		$z\in\bC\setminus(-\infty,0].$
		
		(ii) Given $\lambda\in(0,\infty)$, we can infer from $p$ being real-valued
		that $(\Phi^{-}(\lambda,\cdot),\overline{\Phi^{-}(\lambda,\cdot)})$
		and $(\Phi^{+}(\lambda,\cdot),\overline{\Phi^{+}(\lambda,\cdot)})$
		are also pairs of solutions of $(\tau_{p}-\lambda)f=0$. Identities
		\eqref{id1} are  derived analogously from the Wronskian of the respective
		pairs.

		(iii) Given $\lambda\in(0,\infty)$, identity \eqref{wron3} follows
		from computing the Wronskian of the pair $(\Phi^{+}(\lambda,\cdot),\overline{\Phi^{-}(\lambda,\cdot)})$
		of solutions of $(\tau_{p}-\lambda)f=0$ on  each interval $I_{k}$.
		\qedhere 
	\end{proof}
	\begin{rem*}
		An alternative proof of (ii) is possible by  using the underlying group
		structure explained in Remark \ref{surmk}.
		Let $z=\lambda\in(0,\infty)$. If $k=0$, then $a_{0}^{-}(\lambda)=0$,
		$b_{0}^{-}(\lambda)=1$ and \eqref{id1} trivially holds. Now, suppose
		$1\leq k\leq n$. By taking the determinant of both sides of \eqref{ljprod}
		and with $j=k$, we get 
		\begin{align*}
			|b_{k}^{-}(\lambda)|^{2}-|a_{k}^{-}(\lambda)|^{2}=\gamma_{k}^{-2}\cdot\ldots\cdot\gamma_{1}^{-2}=\tfrac{q_{k}}{q_{0}},
		\end{align*}
		proving \eqref{id1}. The other identity is proved similarly.%
	\end{rem*}
	An immediate consequence of identities \eqref{wron1} and \eqref{hidentity}
	is the following. 
	\begin{cor}
		\label{kappnew} With the notation of Lemma \ref{wronskthm} and for
		$\lambda\in(0,\infty)$ we have %
		\[
		\frac{|b_{0}^{+}(\lambda)|^{2}}{q_{0}^{2}}+\frac{1}{q_{0}q_{n}}=\frac{|a_{0}^{+}(\lambda)|^{2}}{q_{0}^{2}}=\frac{|b_{n}^{-}(\lambda)|^{2}}{q_{n}^{2}}=\frac{1}{q_{0}q_{n}}+\frac{|a_{n}^{-}(\lambda)|^{2}}{q_{n}^{2}}.
		\]
	\end{cor}

	\subsection{The spectral measure\label{subsec:The-spectral-measure}}
	
	We are now ready to discuss some of the spectral properties of $A_{p}$
	and derive a formula for the spectral measure $\mu$ of $A_{p}$.
	We will need the following expressions. An application of \cite[Thm.~7.8]{Weidmann1987spectral}
	to the fundamental system $\Phi$ shows that for $z\in\rho(A_{p})=\bC\setminus\sigma(A_{p}),$
	the resolvent $R_{z}(A_{p})=(A_{p}-z)^{-1}$ can be expressed as integral
	operator 
	\begin{align*}
		R_{z}(A_{p})g(x) & =\int_{\bR}r_{z}(x,y)g(y)\,dy,\quad g\in L^{2}(\bR)
	\end{align*}
	with the\emph{ }resolvent kernel $r_{z}(x,y)$ defined as 
	\begin{align}
		r_{z}(x,y)=\dfrac{1}{W_p(\Phi^{+}(z,\cdot),\Phi^{-}(z,\cdot))}\begin{cases}
			\Phi^{+}(z,x)\Phi^{-}(z,y), & y\leq x,\\
			\Phi^{-}(z,x)\Phi^{+}(z,y), & y>x\label{reskerold}\,.
		\end{cases}
	\end{align}
	Following Weidmann \cite[Sec.~14]{weidmann2003lineare2} %
	we can find $2\times2$ matrices $m^{\pm}$ with entries analytic
	in $\rho(A_{p})$ such that 
	\begin{align}
		r_{z}(x,y)=\begin{cases}
			\overline{\Phi(\overline{z},x)}\cdot m^{+}(z)\Phi(z,y), & y\leq x\\
			\overline{\Phi(\overline{z},x)}\cdot m^{-}(z)\Phi(z,y), & y>x.
		\end{cases}\label{resmat}
	\end{align}
	For any bounded interval $(\gamma,\lambda]$ the spectral measure
	$\mu$ is given by the Weyl-Titchmarsh Formula (cf. \cite[Thm.~XIII.5.18]{Dunford1988linear},
	\cite[Thm.~9.4]{Weidmann1987spectral}, \cite[Thm.~14.5]{weidmann2003lineare2}):
	\begin{align}
		\mu((\gamma,\lambda])=\dfrac{1}{2\pi i}\lim_{\delta\downarrow0}\lim_{\epsilon\downarrow0}\int_{\gamma+\delta}^{\lambda+\delta}\left(m^{+}(t+i\epsilon)-m^{+}(t-i\epsilon)\right)\,dt.\label{wtkform}
	\end{align}
	As $\mu$ is absolutely continuous with respect to the Lebesgue measure
	on $\bR$ this equation can also be written as
	\[
	d\mu\left(\lambda\right)=\dfrac{1}{2\pi
		i}\lim_{\epsilon\downarrow0}\left(m^{+}(\lambda+i\epsilon)-m^{+}(\lambda-i\epsilon)\right)
	\, .
	\]
	The following theorem describes the spectral measure of $A_{p}$. 
	\begin{thm}
		\label{specprojthm} Let $n\in\bN$ and $p$ a step function with
		$n$ jumps. Let $\Phi=(\Phi^{+},\Phi^{-})$ be the fundamental system
		as constructed  in Theorem
		\ref{lemsol} with connection coefficients $a_{k}^{\pm},b_{k}^{\pm}$,
		$0\leq k\leq n$. Define
		\begin{equation}
			\label{eq:cc5}
			\kappa(\sqrt{\lambda})=\frac{|a_{0}^{+}(\lambda)|^{2}}{q_{0}^{2}}=\frac{|b_{n}^{-}(\lambda)|^{2}}{q_{n}^{2}} \geq\frac{1}{q_{0}q_{n}}  \, .
		\end{equation}
		Then the spectral measure of $A_{p}$ is a $2\times2$
		positive matrix measure $\mu$ given by 
		\begin{align}
			d\mu(\lambda)=\dfrac{1}{4\pi\kappa(\sqrt{\lambda})}\begin{bmatrix}\frac{1}{q_{0}} & 0\\
				0 & \frac{1}{q_{n}}
			\end{bmatrix} \dfrac{d\lambda}{\sqrt{\lambda}}
			\,.\label{specmeasdmu}
		\end{align}
		The spectral transform 
		\begin{equation}
			\mathcal{F}_{A_{p}}:L^{2}(\bR)\longrightarrow L^{2}([0,\infty),d\mu),\quad\mathcal{F}_{A_{p}}f(\lambda)=\int_{\bR}f(x)\overline{\Phi(\lambda,x)}\,dx\label{fapnewspec}
		\end{equation}
		yields a spectral representation of $A_{p}$. The inverse $\mathcal{F}_{A_{p}}^{-1}$
		takes the form 
		\begin{align*}
			\mathcal{F}_{A_{p}}^{-1}G(x) & =\int_{0}^{\infty}G(\lambda)\cdot\Phi(\lambda,x)\,d\mu(\lambda)\\
			& =\frac{1}{4\pi}\int_{0}^{\infty}\dfrac{\frac{1}{q_{0}}G_{1}(\lambda)\Phi^{+}(\lambda,x)+\frac{1}{q_{n}}G_{2}(\lambda)\Phi^{-}(\lambda,x)}{\kappa(\sqrt{\lambda})}\,\dfrac{d\lambda}{\sqrt{\lambda}}
		\end{align*}
		for all $G=(G_{1},G_{2})\in L^{2}([0,\infty),d\mu)$. 
	\end{thm}
	
	\begin{proof}
		The result is a direct consequence of Proposition~\ref{weidspecthm}.
		Since  $\Phi$ is  given explicitly in Theorem \ref{lemsol}, 
		we can proceed one step
		further to compute $\mu$.
		
		We follow the lines of \cite{weidmann2003lineare2}. Assume that $\textrm{Im}z>0$.
		By Theorem \ref{lemsol}, $\Phi(z,\cdot)$ is a fundamental system
		of $(\tau_{p}-z)f=0$, so that $\Phi^{+}(z,\cdot)$ lies right and
		$\Phi^{-}(z,\cdot)$ lies left in $L^{2}(\bR)$, respectively. By
		\eqref{ioint} and \eqref{reskerold}, 
		\begin{align}
			r_{z}(x,y) & =\dfrac{iq_{0}}{2a_{0}^{+}(z)\sqrt{z}}\begin{cases}
				\Phi^{+}(z,x)\Phi^{-}(z,y), & y\leq x,\\
				\Phi^{-}(z,x)\Phi^{+}(z,y), & y>x.
			\end{cases}\label{rep2}
		\end{align}
		On the other hand, since $\textrm{Im}\overline{z}<0$, $\overline{\Phi(z,\cdot)}$
		is a fundamental system of $(\tau_{p}-\overline{z})f=0$ where $\overline{\Phi^{+}(z,\cdot)}$
		lies right and $\overline{\Phi^{-}(z,\cdot)}$ lies left in $L^{2}(\bR)$,
		respectively. Therefore %
		\begin{align}
			r_{\overline{z}}(x,y) & =\dfrac{-iq_{n}}{2\overline{b_{n}^{-}(z)\sqrt{z}}}\begin{cases}
				\overline{\Phi^{+}(z,x)}\overline{\Phi^{-}(z,y)}, & y\leq x\\
				\overline{\Phi^{-}(z,x)}\overline{\Phi^{+}(z,y)}, & y>x.
			\end{cases}\label{conjzconj}
		\end{align}
		
		To apply the Weyl-Titchmarsh formula we need to compute the $2\times2$
		matrices $m^{\pm}$ in \eqref{resmat}. This amounts to expressing
		$\Phi^{\pm}(z,x)$ by $\overline{\Phi(\overline{z},x)}$.
		
		Since $m^{\pm}(z)$ in \eqref{resmat} does not depend on $x,y$ , we may
		choose $x,y$  in the unbounded
		intervals. This choice makes the  calculations
		simpler. 
		{} 

		For $x\in I_{n}$ we get 
		\begin{align*}
			\overline{\Phi^{+}(\overline{z},x)} & =e^{-iq_{n}\sqrt{z}x},\quad\overline{\Phi^{-}(\overline{z},x)}=\overline{a_{n}^{-}(\overline{z})}e^{-iq_{n}\sqrt{\overline{z}}x}+\overline{b_{n}^{-}(\overline{z})}e^{iq_{n}\sqrt{\overline{z}}x}
		\end{align*}
		and consequently, 
		\begin{align}
			\Phi^{+}(z,x) & =e^{iq_{n}\sqrt{z}x}\nonumber \\
			& =\dfrac{1}{\overline{b_{n}^{-}(\overline{z})}}\overline{\Phi^{-}(\overline{z},x)}-\dfrac{\overline{a_{n}^{-}(\overline{z})}}{\overline{b_{n}^{-}(\overline{z})}}\overline{\Phi^{+}(\overline{z},x)},\quad x\in I_{n}.\label{eq2}
		\end{align}
		Assuming $y\leq x$ and substituting%
		{} \eqref{eq2} to \eqref{rep2} for Im$z>0$ yields 
		\begin{align}
			r_{z}(x,y) & =\dfrac{iq_{0}}{2a_{0}^{+}(z)\overline{b_{n}^{-}(\overline{z})}\sqrt{z}}\left\{ \overline{\Phi^{-}(\overline{z},x)}-\overline{a_{n}^{-}(\overline{z})}\ \overline{\Phi^{+}(\overline{z},x)}\right\} \Phi^{-}(z,y)\,;\label{rzpos1}
		\end{align}
		If we compare this to \eqref{resmat} we obtain 
		\[
		m^{+}(z)=\dfrac{iq_{0}}{2a_{0}^{+}(z)\overline{b_{n}^{-}(\overline{z})}\sqrt{z}}\begin{bmatrix}0 & -\overline{a_{n}^{-}(\overline{z})}\\
			0 & 1
		\end{bmatrix}\,.
		\]
		The analogous calculation for $x, y \in I_{0},y\leq x,$
		yields
		\[
		m^{+}(\overline{z})=\dfrac{-iq_{n}}{2a_{0}^{+}(\overline{z})\overline{b_{n}^{-}(z)}\sqrt{\overline{z}}}\begin{bmatrix}1 & -b_{0}^{+}(\overline{z})\\
			0 & 0
		\end{bmatrix}.
		\]

		Hence, for $\textrm{Im}z>0$ we have 
		\begin{equation}
			m^{+}(z)-m^{+}(\overline{z})=\begin{bmatrix}\dfrac{iq_{n}}{2a_{0}^{+}(\overline{z})\overline{b_{n}^{-}(z)}\sqrt{\overline{z}}} & \dfrac{-iq_{0}\overline{a_{n}^{-}(\overline{z})}}{2a_{0}^{+}(z)\overline{b_{n}^{-}(\overline{z})}\sqrt{z}}-\dfrac{iq_{n}b_{0}^{+}(\overline{z})}{2a_{0}^{+}(\overline{z})\overline{b_{n}^{-}(z)}\sqrt{\overline{z}}}\\
				0 & \dfrac{iq_{0}}{2a_{0}^{+}(z)\overline{b_{n}^{-}(\overline{z})}\sqrt{z}}
			\end{bmatrix}.\label{specmat}
		\end{equation}
		As $\mu$ is absolutely continuous with respect to the Lebesgue measure
		on $[0,\infty)$, by \eqref{wtkform} the matrix $\mathcal{M}$ of
		densities of $\mu$ has entries 
		\begin{align*}
			\mathcal{M}_{jl}(\lambda) & =\lim_{\epsilon\downarrow0}\dfrac{1}{2\pi i}\left(m_{jk}^{+}(\lambda+i\epsilon)-m_{jk}^{+}(\lambda-i\epsilon)\right).
		\end{align*}
		Applying \eqref{wron1}, \eqref{wron3} and Corollary \ref{kappnew}
		to \eqref{specmat} with $z=\lambda+i\epsilon$, $\epsilon\downarrow0$
		gives 
		\begin{align*}
			\mathcal{M}_{11}(\lambda) & =\dfrac{q_{n}}{4\pi\sqrt{\lambda}a_{0}^{+}(\lambda)\overline{b_{n}^{-}(\lambda)}}=\dfrac{q_{0}}{4\pi\sqrt{\lambda}|a_{0}^{+}(\lambda)|^{2}}=\frac{1}{4\pi q_{0}\kappa\left(\sqrt{\lambda}\right)\sqrt{\lambda}},\\
			\mathcal{M}_{22}(\lambda) & =\frac{q_{0}}{q_{n}}\mathcal{M}_{11}(\lambda),\\
			\mathcal{M}_{12}(\lambda) & =-\dfrac{q_{0}q_{n}\left(\tfrac{1}{q_{n}}\overline{a_{n}^{-}(\lambda)}+\tfrac{1}{q_{0}}b_{0}^{+}(\lambda)\right)}{4\pi a_{0}^{+}(\lambda)\overline{b_{n}^{-}(\lambda)}\sqrt{\lambda}}=0,\\
			\mathcal{M}_{21}(\lambda) & =0.
		\end{align*}
		By definition of $\kappa$ and by Corollary \ref{kappnew}, $\kappa(\sqrt{\lambda})\geq\tfrac{1}{q_{0}q_{n}}>0$
		for all $\lambda\in(0,\infty)$ and 
		\[
		d\mu(\lambda)=\mathcal{M}(\lambda)\,d\lambda=\dfrac{1}{4\pi\kappa(\sqrt{\lambda})}\begin{bmatrix}\frac{1}{q_{0}} & 0\\
			0 & \frac{1}{q_{n}}
		\end{bmatrix}\dfrac{d\lambda}{\sqrt{\lambda}}.
		\]
		The formulas for the spectral transform and its inverse are stated  in Theorem \ref{weidspecthm}. 
	\end{proof}
	\begin{rem}
		\label{bddrmk}The substitution $\lambda=u^{2}$ yields the following
		expressions for the spectral matrix and the inverse spectral Fourier
		transform: 
		\begin{align}
			d\mu(u^{2}) & =\dfrac{1}{2\pi\kappa(u)}\begin{bmatrix}\frac{1}{q_{0}} & 0\\
				0 & \frac{1}{q_{n}}
			\end{bmatrix}\,du,\label{specmeasdmu2}\\
			\mathcal{F}_{A_{p}}^{-1}G(x) & =\frac{1}{2\pi}\int_{0}^{\infty}\dfrac{\frac{1}{q_{0}}G_{1}(u^{2})\Phi^{+}(u^{2},x)+\frac{1}{q_{n}}G_{2}(u^{2})\Phi^{-}(u^{2},x)}{\kappa(u)}\,du.\nonumber 
		\end{align}
		In addition, since $a^+$ and $b^-$ are almost periodic trigonometric
		polynomials, 
		there exist $r\in\bN$, which  increases with the number of jumps of $p$, and constants $c_{0},\ldots,c_{r},\lambda_{1},\ldots,\lambda_{r}\in\bR$
		such that 
		\begin{align*}
			\kappa(u)=\frac{|a_0^+(u^2)|}{q_0^2} = c_{0}+\sum_{j=1}^{r}c_{j}\cos(\lambda_{j}u),\quad u\in(0,\infty).
		\end{align*}
	\end{rem}

	\subsection{The reproducing kernel of $PW_{\Lambda}(A_{p})$\label{subsec:The-reproducing-kernel}}
	Next we investigate the reproducing kernel of the Paley-Wiener space
	$PW_{\Lambda}(A_{p}) $. With the spectral measure of $A_p$ in place, the general
	properties of $PW_{\Lambda}(A_{p})  $ follow exactly as for the classical Paley-Wiener
	space. 
	Recall that by definition, $PW_{\Lambda}(A_{p})=\chi_{\Lambda}(A_{p})(L^{2}(\bR))$. 

	\begin{thm}
		\label{thmunifbdd} Let $\Lambda\subset\bR_{0}^{+}$ be a Borel set
		of finite measure and $p$ a piecewise constant function. Define $\Lambda^{1/2}=\{\lambda\in\bR_{0}^{+}:\lambda^{2}\in\Lambda\}$
		and 
		\begin{equation} \label{eq:cc6}
			\vartheta(u,x,y)  =\frac{1}{q_{0}}\overline{\Phi^{+}(u^{2},x)}\Phi^{+}(u^{2},y)+\frac{1}{q_{n}}\overline{\Phi^{-}(u^{2},x)}\Phi^{-}(u^{2},y).
		\end{equation}
		Then $PW_{\Lambda}(A_{p})$ is a reproducing kernel Hilbert space
		with kernel 
		\begin{align}
			k_{\Lambda}(x,y)=\dfrac{1}{2\pi}\int_{\Lambda^{1/2}}\dfrac{\vartheta(u,x,y)}{\kappa(u)}\,du,\quad x,y\in\bR.\label{kref}
		\end{align}
	\end{thm}
	
	\begin{proof}
		We first claim that $\{\|\Phi(\cdot,x)\|_{L^{2}(\Lambda,d\mu)}\}_{x\in\bR}$
		is bounded. Indeed, Theorem \ref{specprojthm} implies that for each
		$x\in\bR$, 
		\begin{align*}
			\|\Phi(\cdot,x)\|_{L^{2}(\Lambda,d\mu)}^{2} & =\frac{1}{2\pi}\int_{\Lambda^{1/2}}\frac{\frac{1}{q_{0}}|\Phi^{+}(u^{2},x)|^{2}+\frac{1}{q_{n}}|\Phi^{-}(u^{2},x)|^{2}}{\kappa(u)}\,du\\
			& \leq\frac{1}{2\pi}\essupp_{u\in\Lambda^{1/2}}|\Phi^{\pm}(u^{2},x)|^{2}\int_{\Lambda^{1/2}}\frac{q_{0}+q_{n}}{q_{0}q_{n}\kappa(u)}\,du.
		\end{align*}
		By \eqref{specmeasdmu} $\left(q_{0}q_{n}\kappa\left(\sqrt{\lambda}\right)\right)^{-1}\leq1$,
		and Lemma \ref{unifbdd} asserts the uniform boundedness of $\Phi^{\pm}$
		, so 
		\begin{align*}
			\|\Phi(\cdot,x)\|_{L^{2}(\Lambda,d\mu)} & \leq\essupp_{u\in\Lambda^{1/2}}|\Phi^{\pm}(u^{2},x)|\left(\frac{q_{0}+q_{n}}{2\pi}|\Lambda^{1/2}|\right)^{1/2}=C<\infty.
		\end{align*}
		for every  $x\in\bR$.
		
		Then by the claim and by the unitarity of
		$\mathcal{F}_{A_{p}}$, we obtain
		\begin{align*}
			\left|\int_{\Lambda}\mathcal{F}_{A_{p}}f(\lambda)\cdot\Phi(\lambda,x)\,d\mu(\lambda)\right|
			&
			\leq\|\Phi(\cdot,x)\|_{L^{2}(\Lambda,d\mu)}\|\mathcal{F}_{A_{p}}f\|_{L^{2}(\Lambda,d\mu)}\\
			&=|\Phi(\cdot,x)\|_{L^{2}(\Lambda,d\mu)}\|f\|_{2}\leq
			C\|f\|_{2} \, ,
		\end{align*}
		and the integral makes sense for every $x\in \bR $ and for all $f\in
		PW_{\Lambda}(A_{p})$. Furthermore, 
		\begin{align*}
			f(x)=\chi_{\Lambda}(A_{p})f(x) & =\int_{\Lambda}\mathcal{F}_{A_{p}}f(\lambda)\cdot\Phi(\lambda,x)\,d\mu(\lambda)\\
			& =\int_{\Lambda}\int_{\bR}f(y)\overline{\Phi(\lambda,y)}\cdot\Phi(\lambda,x)\,dy\,d\mu(\lambda)\\
			& =\int_{\bR}f(y)\int_{\Lambda}\overline{\Phi(\lambda,y)}\cdot\Phi(\lambda,x)\,d\mu(\lambda)\,dy\\
			& =\int_{\bR}f(y)\overline{k_{\Lambda}(x,y)}\,dy.
		\end{align*}
		Consequently  the
		evaluation map $f\mapsto f(x)$, $f\in PW_{\Lambda}(A_{p})$ is bounded,
		in other words,  $PW_{\Lambda}(A_{p})$ is a reproducing kernel Hilbert space.
		Formula \eqref{kref} now follows by substituting the expressions
		in Theorem \ref{lemsol} and Theorem \ref{specprojthm} to \eqref{kernelformula}. 
		See also  \cite[Thm.~XIII.5.24]{Dunford1988linear}).
	\end{proof}
	
	\subsection{Computation of the reproducing kernel\label{subsec:Computation-of-the}}
	
	In view of eventual numerical implementations, it may be helpful to
	make the structure of $k_\Lambda $ even  more explicit. By
	Theorem~\ref{thmunifbdd} $k_\Lambda $ depends the fundamental
	solutions $\Phi ^{\pm}$. 

	Fix $x,y\in\bR$. Since the connection
	coefficients $a^{\pm}(u^{2})$ and $b^{\pm}(u^{2})$ are almost periodic
	polynomials in $u$, there exist a positive integer $m(x,y)$ and
	real numbers $\alpha_{k}(x,y),\beta_{k}(x,y),$ $1\leq k\leq m(x,y)$
	such that 
	\begin{align}
		\vartheta(u,x,y)=\sum_{k=1}^{m(x,y)}\alpha_{k}(x,y)e^{i\beta_{k}(x,y)u}.\label{almvartheta}
	\end{align}
	Note that for $0\leq j,l\leq n$ fixed and $x\in I_{j},y\in
        I_{l}$, the coefficients $\alpha_{k}(x,y)$ depend only on $j$
        and $l$, and are thus constant on $I_j\times I_l$. We can
        therefore write
        $$
        \vartheta(u,x,y)=\sum_{j,l=1}^n \vartheta _{jl}(u) \chi _{I_j}(x)\chi _{I_l}(y)
        $$
with  almost periodic functions $\vartheta _{jl}$. This is the
formulation of Theorem~\ref{kernintro} of the introduction.

Furthermore, by substituting~\eqref{pw1} and \eqref{pw2}  in the definition of $\vartheta
$, we see that the exponents $\beta_{k}(x,y)$, $x\in I_{j}$,
	$y\in I_{l}$ are of the form 
	\begin{align}
		\beta_{k}(x,y)=c_{k}^{(j,l)}\pm q_{j}x\pm q_{l}y,\label{eq:expo}
	\end{align}
	where the scalars $c_{k}^{(j,l)}\in\bR$ depend on the jump positions $\{t_{r}\}_{r=1}^{n}$
	and the local parameters $\{q_{r}\}_{r=0}^{n}$. %

	Since $\kappa$ is bounded below on $(0,\infty)$ and $\Lambda$ has
	finite Lebesgue measure, the integral 
	\begin{align}
		J(s)=\frac{1}{2\pi}\int_{\Lambda^{1/2}}\frac{e^{isu}}{\kappa(u)}\,du,\quad s\in\bR
	\end{align}
	is a well-defined function whose Fourier transform is supported in
	$\overline{\Lambda^{1/2}}$. By \eqref{kref} and \eqref{almvartheta},
	we can now write the reproducing kernel $k_{\Lambda}$ as 
	\begin{align}
		k_{\Lambda}(x,y) & =\frac{1}{2\pi}\sum_{k=1}^{m(x,y)}\alpha_{k}(x,y)\int_{\Lambda^{1/2}}\frac{e^{i\beta_{k}(x,y)u}}{\kappa(u)}\,du=\sum_{k=1}^{m(x,y)}\alpha_{k}(x,y)J(\beta_{k}(x,y)).\label{eq:kerexpo}
	\end{align}
	The function $J$ is the Fourier transform of $\kappa $ occuring in the
	spectral measure and usually has to be computed numerically or by
	means of the series expansion (see the next section).
	
	\section{Concrete examples}
	
	\label{sec:ex}
	
	In this section, we derive explicit expressions for the reproducing
	kernel for the case $n=2$. A treatment of the case $n=1$ (two intervals of constant bandwidth)
	with a formula for the reproducing kernel and a sampling theorem  can be
	found in \cite[Sec.~4]{cpam}.

	\subsection{The case $n=2$: three intervals with constant local bandwidth}
	
	\label{secthreeex} We consider a step function $p$ with two jumps. 
	Without loss of generality we assume that  the middle interval is
	centered at the origin. We first determine the density function
	$\kappa$ in \eqref{eq:cc5} of Theorem~\ref{specprojthm}. 
	\begin{lem}
		\label{Jlem} Let $p_{0},p_{1},p_{2},T>0$ and $\Lambda\subset\bR_{0}^{+}$
		be of finite measure and set $q_{k}=p_{k}^{-1/2},k=0,1,2$. Set 
		\begin{align}
			p(x)&= p_0 \chi _{(-\infty , -T/2]}(x)  + p_1 \chi _{(-T/2,T/2]}(x)+
			p_2     \chi _{[(T/2, \infty )}(x) \, , \label{pthree} \\ 
			C & =\tfrac{1}{16q_{0}^{2}}\left[\left(1+\tfrac{q_{0}}{q_{1}}\right)^{2}\left(1+\tfrac{q_{1}}{q_{2}}\right)^{2}+\left(1-\tfrac{q_{0}}{q_{1}}\right)^{2}\left(1-\tfrac{q_{1}}{q_{2}}\right)^{2}\right],\label{constC}\\
			K & =\tfrac{1}{8q_{0}^{2}}\left(1-\tfrac{q_{0}^{2}}{q_{1}^{2}}\right)\left(1-\tfrac{q_{1}^{2}}{q_{2}^{2}}\right),\label{constK}\\
			\zeta & =2q_{1}T \, . 
		\end{align}
		Then the density in the spectral measure in \eqref{eq:cc5} is $\kappa
		(u) = C+K\cos\zeta u$ and the associated integral is 
		\[
		J(s)=\frac{1}{2\pi}\int_{\Lambda^{1/2}}\dfrac{e^{isu}}{C+K\cos\zeta u}\,du,\quad s\in\bR.
		\]
	\end{lem}
	
	\begin{proof}
		By Theorem \ref{lemsol}, the fundamental system $\Phi(u^{2},\cdot)=(\Phi^{+}(u^{2},\cdot),\Phi^{-}(u^{2},\cdot)),u\in(0,\infty)$
		is of the form 
		\begin{align*}
			\Phi^{+}(u^{2},x) & =\begin{cases}
				a_{0}^{+}(u^{2})e^{iq_{0}xu}+b_{0}^{+}(u^{2})e^{-iq_{0}xu}, & x\in(-\infty,-\frac{T}{2}]\\
				a_{1}^{+}(u^{2})e^{iq_{1}xu}+b_{1}^{+}(u^{2})e^{-iq_{1}xu}, & x\in(-\frac{T}{2},\frac{T}{2}]\\
				e^{iq_{2}xu}, & x\in(\frac{T}{2},\infty)
			\end{cases}\\
			\Phi^{-}(u^{2},x) & =\begin{cases}
				e^{-iq_{0}xu}, & x\in(-\infty,-\frac{T}{2}]\\
				a_{1}^{-}(u^{2})e^{iq_{1}xu}+b_{1}^{-}(u^{2})e^{-iq_{1}xu}, & x\in(-\frac{T}{2},\frac{T}{2}]\\
				a_{2}^{-}(u^{2})e^{iq_{2}xu}+b_{2}^{-}(u^{2})e^{-iq_{2}xu}, & x\in(\frac{T}{2},\infty).
			\end{cases}
		\end{align*}
		Applying \eqref{reccoefrj} for $l = 2$ yields
		\begin{align*}
			\begin{bmatrix}a_{0}^{+}(u^{2})\\
				b_{0}^{+}(u^{2})
			\end{bmatrix} & =\dfrac{1}{4}\begin{bmatrix}\left(1+\frac{q_{0}}{q_{1}}\right)\left(1+\frac{q_{1}}{q_{2}}\right)e^{i\frac{T}{2}\left(q_{0}-2q_{1}+q_{2}\right)u}+\left(1-\frac{q_{0}}{q_{1}}\right)\left(1-\frac{q_{1}}{q_{2}}\right)e^{i\frac{T}{2}\left(q_{0}+2q_{1}+q_{2}\right)u}\\
				\left(1-\frac{q_{0}}{q_{1}}\right)\left(1+\frac{q_{1}}{q_{2}}\right)e^{i\frac{T}{2}\left(-q_{0}-2q_{1}+q_{2}\right)u}+\left(1+\frac{q_{0}}{q_{1}}\right)\left(1-\frac{q_{1}}{q_{2}}\right)e^{i\frac{T}{2}\left(-q_{0}+2q_{1}+q_{2}\right)u}
			\end{bmatrix} 
		\end{align*}
		Since $|Ae^{i\alpha} + Be^{i\beta}|^2 = A^2 + B^2 + 2AB \cos (\alpha -
		\beta )$, we obtain the claimed formula for $\kappa $ from the above
		as 
		\begin{align*}
			\kappa(u)&=\tfrac{|a_{0}^{+}(u^{2})|^{2}}{q_{0}^{2}} \\
			&= \tfrac{1}{16q_0^2}\left[( 1+\tfrac{q_0}{q_1} \big)^2 \big( 1+\tfrac{q_1}{q_2} \big)^2 +
			\big( 1-\tfrac{q_0}{q_1}\big)^2 \big(  1-\tfrac{q_1}{q_2}\big)^2\right]\ +  
			\tfrac{1}{8q_0^2} \big( 1-\tfrac{q_0^2}{q_1^2} \big) \big( 1-\tfrac{q_1^2}{q_2^2} \big)
			\cos (2Tq_1u)\\
			&  =C+K\cos\zeta u \, ,
		\end{align*}
		with the constants $C$ and $K$ from \eqref{constC} and
		\eqref{constK}. Consequently, 
		\begin{align}
			J(s)=\dfrac{1}{2\pi}\int_{\Lambda^{1/2}}\frac{e^{isu}}{C+K\cos\zeta u}\,du\label{eq:J}
		\end{align}
		for any $s\in\bR$. 
	\end{proof}
	To the best of our knowledge, $J$ does not belong to any known  class of
	special functions. If the spectrum $\Lambda $ is an interval, we can
	expand $J$ into a series
	in terms of the cardinal sine function $\sinc\,x = \frac{\sin x}{x}$. Its partial sums
	converge  to $J$ uniformly on $\bR$ and at a geometric rate. 
	\begin{thm}
		\label{jformthm} Let $R=\frac{K}{C}$ and $\Lambda = [0,\Omega ]$. Then for $s\in\bR$, 
		\begin{align*}
			J(s)=\dfrac{\Omega^{1/2}}{2C\pi}\sum_{m=0}^{\infty}\sum_{l=0}^{m}\left(-\frac{R}{2}\right)^{m}\binom{m}{l}e^{i\tfrac{\Omega^{1/2}}{2}(s+(m-2l)\zeta)}\sinc\left(\tfrac{\Omega^{1/2}}{2}(s+(m-2l)\zeta)\right).
		\end{align*}
		Moreover, the $M$-th partial sum $J_{M}$ 
		of $J$ satisfies the error estimate 
		\begin{align*}
			\sup_{s\in\bR}|J(s)-J_{M}(s)|\leq\dfrac{\Omega^{1/2}}{2C\pi}\dfrac{|R|^{M+1}}{1-|R|}.
		\end{align*}
	\end{thm}
	
	Note that by the definition of $K$ and $C$  we have $|R| = |K|/C
	<1$ so that the integrand of $J$ does not have any
	singularities. Therefore the error estimate implies uniform
	convergence of the partial 
	sums and this rate depends only on $R$, which in turn depends only on
	the parameters of the local bandwidths $p$. 
	\begin{proof}
		By assumption we have $|R|<1$. The expression for $J(s)$, $s\in\bR$,
		can be expanded in a geometric series. 
		\begin{align}
			\qquad J(s) & =\dfrac{1}{2\pi
				C}\int_{0}^{\Omega^{1/2}}\dfrac{e^{isu}}{1-\left(-R\cos\zeta
				u\right)}\,du=\dfrac{1}{2\pi C}\sum_{m=0}^{\infty}\left(-R\right)^{m}\int_{0}^{\Omega^{1/2}}e^{isu}\cos^{m}\zeta u\,du\,.\label{eq:jm}
		\end{align}
		The interchange of the above integral and sum follows from Weierstrass
		$M$-test. For $m\in\bN_{0}$, define the bandlimited function 
		\[
		F_{m}(s)=\int_{0}^{\Omega^{1/2}}e^{isu}\cos^{m}\zeta u\,du,\quad s\in\bR
		\]
		so that $J=\frac{1}{2\pi C}\sum_{m=0}^{\infty}\left(-R\right)^{m}F_{m}$.
		To compute $F_{m}$, we write cosine using complex exponentials: 
		\begin{align}
			F_{m}(s) & =\int_{0}^{\Omega^{1/2}}e^{isu}\cdot\left(\dfrac{e^{i\zeta u}+e^{-i\zeta u}}{2}\right)^{m}\,du\nonumber \\
			& =\frac{1}{2^{m}}\sum_{l=0}^{m}\binom{m}{l}\int_{0}^{\Omega^{1/2}}e^{i(s+(m-2l)\zeta)u}\,du\nonumber \\
			& =\frac{\Omega^{1/2}}{2^{m}}\sum_{l=0}^{m}\binom{m}{l}e^{i\tfrac{\Omega^{1/2}}{2}(s+(m-2l)\zeta)}\sinc\left(\tfrac{\Omega^{1/2}}{2}(s+(m-2l)\zeta)\right).\label{fm2}
		\end{align}
		Substituting \eqref{fm2} to \eqref{eq:jm} gives the desired expansion.
		Now, since $|F_{m}(s)|\leq\Omega^{1/2}$ for all $m\in\bN_{0}$, then
		for $M\in\bN_{0}$ we observe that 
		\begin{align*}
			|J(s)-J_{M}(s)|=\frac{1}{2\pi
				C}\left|\sum_{m=M+1}^{\infty}(-R)^{m}F_{m}(s)\right|\leq\frac{\Omega^{1/2}}{2\pi
				C}\sum_{m=M+1}^{\infty}|R|^{m}=\frac{\Omega^{1/2}}{2\pi
				C}\frac{|R|^{M+1}}{1-|R|}.
		\end{align*}
		Taking the supremum over all $s\in\bR$ completes the proof. 
	\end{proof}
	
		\begin{figure}
		\centering \includegraphics[width=0.95\linewidth]{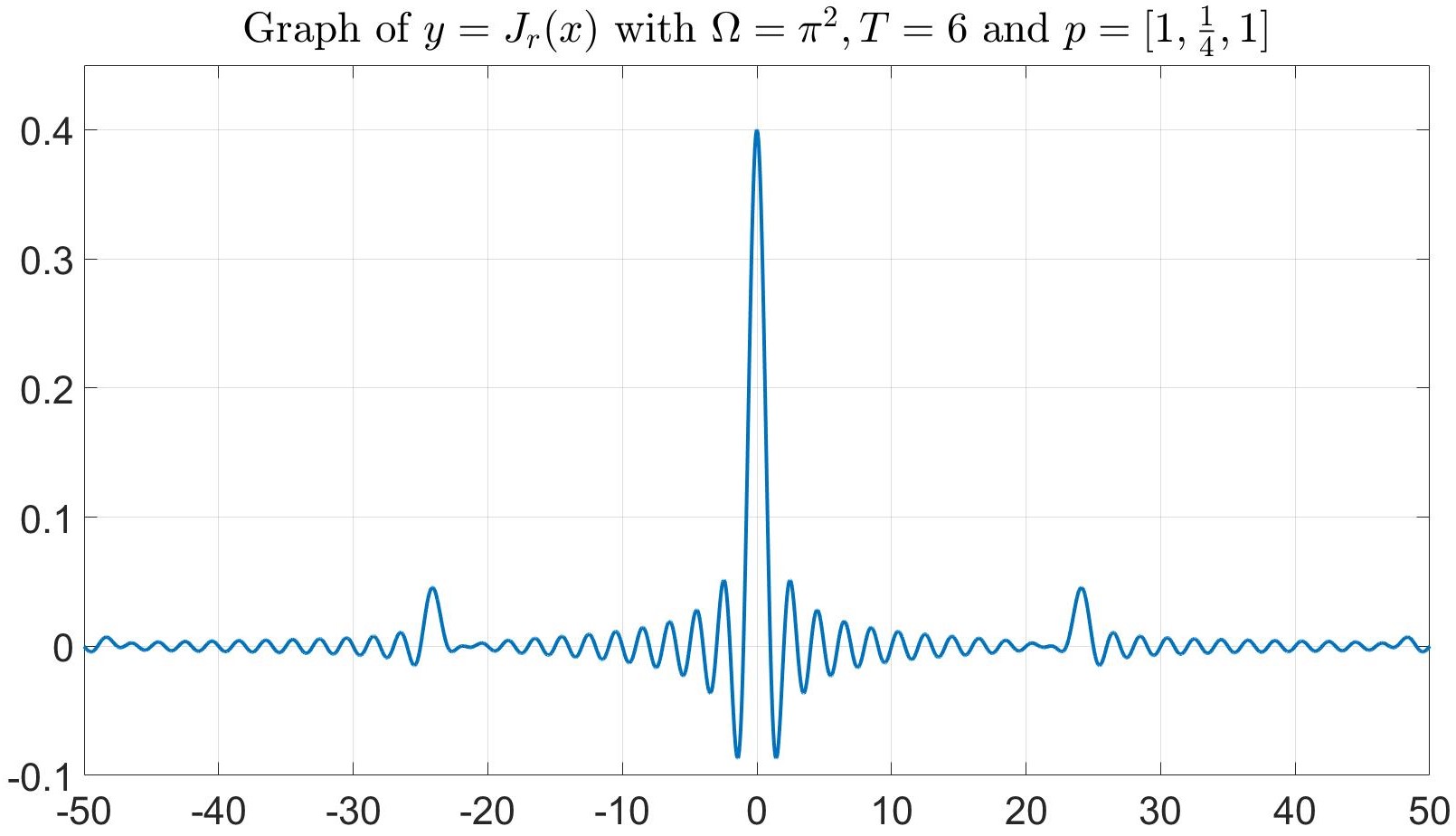} \caption{Graph of $J_{r}$ when $\Lambda=[0,\pi^{2}]$, $T=6$, $p_{0}=p_{2}=1$
			and $p_{1}=\frac{1}{4}$ on the interval $[-50,50]$. }
		\label{Jgraph}
	\end{figure}
	
	\begin{rem}
		\label{Jrealrem}
		(i) Using the double angle identity for the sine function, one can show
		that the real part $J_{r}$ of $J$ is given by 
		\begin{align}
			J_{r}(s) & =\dfrac{\Omega^{1/2}}{2C\pi}\sum_{m=0}^{\infty}\sum_{l=0}^{m}\left(-\frac{R}{2}\right)^{m}\binom{m}{l}\sinc\left(\Omega^{1/2}(s+(m-2l)\zeta)\right),\quad s\in\bR.\label{eq:Jr}
		\end{align}
		
		(ii) In the special case $\Omega ^{1/2} = 2\pi \zeta l = 4\pi q_1T l$ for
		some $l\in \bN$, i.e., if the bandwidth is correlated to the
		parameters of $p$, then 
		a closed formula for $J$  can be derived in terms of special
		functions. See \cite[Thms.~B.0.4, B.0.5]{CelizDissertation}.
		
		(iii) In the general case, an alternative expansion for $J_{r}$
		can be derived by a change of variables $m-2l\to k$ in \eqref{eq:Jr}
		and a change of the summation order. 
		With %
		\begin{align}
			c_{k} &=\frac{\Omega^{1/2}}{2C\pi}\sum_{j=0}^{\infty}\binom{2j+|k|}{j}\left(-\frac{R}{2}\right)^{2j+|k|}\\
			&=\frac{\Omega^{1/2}}{2C\pi}\left(-\frac{R}{2}\right)^{|k|}\,_{2}F_{1}\left(\frac{|k|+1}{2},\frac{|k|+2}{2};|k|+1;R^{2}\right),\quad k\in\bZ,\label{ckhyp}
		\end{align}
		we obtain 
		\begin{align*}
			J_{r}(s)=\sum_{k=-\infty}^{\infty}c_{k}\sinc(\Omega^{1/2}(s-k\zeta)).
		\end{align*}
	
		If $s\in\bigcap_{k\in\bZ}\left(k\zeta+\tfrac{\pi}{\Omega^{1/2}}\bZ^{*}\right)$,
		then $J_{r}(s)=0$. In Figure \ref{Jgraph} we see that for the indicated
		set of parameters, the zeros of $J_{r}$ are all integers that are
		not a multiple of $\zeta=2\cdot2\cdot6=24$ where peaks with decaying
		heights occur. 
	\end{rem}
	
	We now derive a complete formula for the reproducing kernel of $PW_{[0,\Omega]}(A_{p})$
	with $p$ given in \eqref{pthree}. 
	\begin{thm}
		\label{ker3pardami}
		Let $p_{0},p_{1},p_{2},T>0$, $\Lambda = [0,\Omega]
		\subset\bR_{0}^{+}$  and  
		$ p= p_0 \chi _{(-\infty , -T/2]}  + p_1 \chi _{(-T/2,T/2]}+
		p_2     \chi _{[(T/2, \infty )}$. 
		Then the reproducing kernel of
		$PW_{[0,\Omega]}(A_{p})$ is of the form 
		\begin{align*}
			k_{\Lambda}(x,y)=\sum_{j,l=0}^{2}k_{{jl}}(x,y)\chi_{I_{j}}(x)\chi_{I_{l}}(y),\quad x,y\in\bR.
		\end{align*}
		The functions $k_{{jl}}$ are 
		\begingroup
		\allowdisplaybreaks
		\begin{flalign*}
			k_{00}(x,y) & =\tfrac{q_{0}\Omega^{1/2}}{\pi}\sinc(q_{0}\Omega^{1/2}(x-y))+\tfrac{1}{4q_{0}}\left(1-\tfrac{q_{0}^{2}}{q_{1}^{2}}\right)\left(1+\tfrac{q_{1}^{2}}{q_{2}^{2}}\right)J_{r}(q_{0}(x+y+T))\\
			& +\tfrac{1}{8q_{0}}\left(1-\tfrac{q_{0}}{q_{1}}\right)^{2}\left(1-\tfrac{q_{1}^{2}}{q_{2}^{2}}\right)J_{r}(q_{0}(x+y+T)+2q_{1}T)\\
			&
			+\tfrac{1}{8q_{0}}\left(1+\tfrac{q_{0}}{q_{1}}\right)^{2}\left(1-\tfrac{q_{1}^{2}}{q_{2}^{2}}\right)J_{r}(q_{0}(x+y+T)-2q_{1}T)\\
			k_{11}(x,y) & =\tfrac{1}{2}\left[\tfrac{1}{q_{0}}\left(1+\tfrac{q_{1}^{2}}{q_{2}^{2}}\right)+\tfrac{1}{q_{2}}\left(1+\tfrac{q_{1}^{2}}{q_{0}^{2}}\right)\right]J_{r}(q_{1}(x-y))+\tfrac{1}{2q_{0}}\left(1-\tfrac{q_{1}^{2}}{q_{2}^{2}}\right)J_{r}(q_{1}(x+y-T))\\
			& +\tfrac{1}{2q_{2}}\left(1-\tfrac{q_{1}^{2}}{q_{0}^{2}}\right)J_{r}(q_{1}(x+y+T)),\\
			k_{22}(x,y) & =\tfrac{q_{2}\Omega^{1/2}}{\pi}\sinc(q_{2}\Omega^{1/2}(x-y))+\tfrac{1}{4q_{2}}\left(1+\tfrac{q_{1}^{2}}{q_{0}^{2}}\right)\left(1-\tfrac{q_{2}^{2}}{q_{1}^{2}}\right)J_{r}(q_{2}(x+y-T))\\
			& +\tfrac{1}{8q_{2}}\left(1-\tfrac{q_{1}^{2}}{q_{0}^{2}}\right)\left(1-\tfrac{q_{2}}{q_{1}}\right)^{2}J_{r}(q_{2}(x+y-T)-2q_{1}T)\\
			& +\tfrac{1}{8q_{2}}\left(1-\tfrac{q_{1}^{2}}{q_{0}^{2}}\right)\left(1+\tfrac{q_{2}}{q_{1}}\right)^{2}J_{r}(q_{2}(x+y-T)+2q_{1}T),\\
			k_{01}(x,y) & =k_{10}(y,x)=\tfrac{1}{4q_{0}}\left(1+\tfrac{q_{0}}{q_{1}}\right)\left(1+\tfrac{q_{1}}{q_{2}}\right)^{2}J_{r}(q_{0}(x+\tfrac{T}{2})-q_{1}(y+\tfrac{T}{2}))\\
			& +\tfrac{1}{4q_{0}}\left(1-\tfrac{q_{0}}{q_{1}}\right)\left(1-\tfrac{q_{1}}{q_{2}}\right)^{2}J_{r}(q_{0}(x+\tfrac{T}{2})+q_{1}(y+\tfrac{T}{2}))\\
			&
			+\tfrac{1}{4q_{0}}\left(1+\tfrac{q_{0}}{q_{1}}\right)\left(1-\tfrac{q_{1}^{2}}{q_{2}^{2}}\right)J_{r}(q_{0}(x+\tfrac{T}{2})+q_{1}(y-\tfrac{3T}{2}))\\
			& +\tfrac{1}{4q_{0}}\left(1-\tfrac{q_{0}}{q_{1}}\right)\left(1-\tfrac{q_{1}^{2}}{q_{2}^{2}}\right)J_{r}(q_{0}(x+\tfrac{T}{2})-q_{1}(y-\tfrac{3T}{2})),\\
			k_{02}(x,y) & =k_{20}(y,x)=\tfrac{1}{2q_{0}}\left(1-\tfrac{q_{0}}{q_{1}}\right)\left(1-\tfrac{q_{1}}{q_{2}}\right)J_{r}(q_{0}(x+\tfrac{T}{2})+q_{1}T-q_{2}(y-\tfrac{T}{2}))\\
			& +\tfrac{1}{2q_{0}}\left(1+\tfrac{q_{0}}{q_{1}}\right)\left(1+\tfrac{q_{1}}{q_{2}}\right)J_{r}(q_{0}(x+\tfrac{T}{2})-q_{1}T-q_{2}(y-\tfrac{T}{2})),\\
			k_{12}(x,y) & =k_{21}(y,x)=\tfrac{1}{4q_{2}}\left(1+\tfrac{q_{1}}{q_{0}}\right)^{2}\left(1+\tfrac{q_{2}}{q_{1}}\right)J_{r}(q_{1}(x-\tfrac{T}{2})-q_{2}(y-\tfrac{T}{2}))\\
			& +\tfrac{1}{4q_{2}}\left(1-\tfrac{q_{1}}{q_{0}}\right)^{2}\left(1-\tfrac{q_{2}}{q_{1}}\right)J_{r}(q_{1}(x-\tfrac{T}{2})+q_{2}(y-\tfrac{T}{2}))\\
			&
			+\tfrac{1}{4q_{2}}\left(1-\tfrac{q_{1}^{2}}{q_{0}^{2}}\right)\left(1+\tfrac{q_{2}}{q_{1}}\right)J_{r}(q_{1}(x+\tfrac{3T}{2})+q_{2}(y-\tfrac{T}{2}))\\
			& +\tfrac{1}{4q_{2}}\left(1-\tfrac{q_{1}^{2}}{q_{0}^{2}}\right)\left(1-\tfrac{q_{2}}{q_{1}}\right)J_{r}(q_{1}(x+\tfrac{3T}{2})-q_{2}(y-\tfrac{T}{2})).
		\end{flalign*}
	\endgroup
	\end{thm}
	
	\begin{proof}
		Initially we have to consider the following nine pairs of indices: 
		\[
		(j,l)\in\{(0,0),(0,1),(0,2),(1,0),(1,1),(1,2),(2,0),(2,1),(2,2)\}.
		\]
		By the symmetry of the reproducing kernel $k_{\Lambda}(x,y)=\overline{k_{\Lambda}(y,x)}$
		for all $x,y\in\bR$, we see that $(1,0)$ follows from $(0,1)$,
		$(2,0)$ follows from $(0,2)$, and $(2,1)$ follows from $(1,2)$.
		Furthermore, since the knots of $p$ are symmetric at the origin,
		$(2,2)$ follows from $(0,0)$ and $(1,2)$ follows from $(0,1)$
		by applying the replacement rule $(x,y,T,q_{0},q_{1},q_{2})\to(y,x,-T,q_{2},q_{1},q_{0})$.
		Thus, it suffices to take the cases $(j,l)\in\{(0,0),(0,1),(0,2),(1,1)\}.$ 
		
		Suppose $(j,l)=(0,0)$, i.e., $x,y\leq-\tfrac{T}{2}$. Then upon rearranging
		the terms, we obtain  
		\begin{align}
			\vartheta(u,x,y) &  =\frac{1}{q_{0}}\overline{\Phi^{+}(u^{2},x)}\Phi^{+}(u^{2},y)+\frac{1}{q_{n}}\overline{\Phi^{-}(u^{2},x)}\Phi^{-}(u^{2},y),\\
			&=\tfrac{1}{q_{0}}(\overline{a_{0}^{+}(u^{2})}e^{-iq_{0}xu}+\overline{b_{0}^{+}(u^{2})}e^{iq_{0}xu})(a_{0}^{+}(u^{2})e^{iq_{0}yu}+b_{0}^{+}(u^{2})e^{-iq_{0}yu}) +\tfrac{1}{q_{2}}e^{iq_{0}xu}e^{-iq_{0}yu}\nonumber \\
			& =\tfrac{|a_{0}^{+}(u^{2})|^{2}}{q_{0}}e^{-iq_{0}(x-y)u}+\left(\tfrac{1}{q_{2}}+\tfrac{|b_{0}^{+}(u^{2})|^{2}}{q_{0}}\right)e^{iq_{0}(x-y)u}\nonumber \\
			& +\tfrac{1}{q_{0}}\left(a_{0}^{+}(u^{2})\overline{b_{0}^{+}(u^{2})}e^{iq_{0}(x+y)u}+\overline{a_{0}^{+}(u^{2})}b_{0}^{+}(u^{2})e^{-iq_{0}(x+y)u}\right).
		\end{align}
		The above expression simplifies, because 
		$\frac{|b_{0}^{+}(\lambda)|^{2}}{q_{0}^{2}}+\frac{1}{q_{0}q_{n}}=\frac{|a_{0}^{+}(\lambda)|^{2}}{q_{0}^{2}}$
		by   Corollary \ref{kappnew}
		and by  the definition of $\kappa = |a^+_0(u)|^2/q_0^2$ in
		\eqref{specmeasdmu}. Upon inserting the connection coefficients
		$a_0^+$ and $b_0^+$ from the proof of Lemma~\ref{Jlem}, we obtain  
		\begin{align*}
			\vartheta(u,x,y) & =2q_{0}\kappa(u)\cos(q_{0}(x-y)u)  +\tfrac{1}{4q_{0}}\left(1+\tfrac{q_{1}^{2}}{q_{2}^{2}}\right)\left(1-\tfrac{q_{0}^{2}}{q_{1}^{2}}\right)\cos(q_{0}(x+y+T)u)\\
			& +\tfrac{1}{8q_{0}}\left(1-\tfrac{q_{1}^{2}}{q_{2}^{2}}\right)\left(1-\tfrac{q_{0}}{q_{1}}\right)^{2}\cos((q_{0}(x+y+T)+2q_{1}T)u)\\
			& +\tfrac{1}{8q_{0}}\left(1-\tfrac{q_{1}^{2}}{q_{2}^{2}}\right)\left(1+\tfrac{q_{0}}{q_{1}}\right)^{2}\cos((q_{0}(x+y+T)-2q_{1}T)u).
		\end{align*}
		Hence, 
		\begin{align*}
			k_{\Lambda}(x,y) & =\dfrac{1}{2\pi}\int_{\Lambda^{1/2}}\dfrac{\vartheta(u,x,y)}{\kappa(u)}\,du\\
			& =\tfrac{q_{0}\Omega^{1/2}}{\pi}\sinc(q_{0}\Omega^{1/2}(x-y)) +\tfrac{1}{4q_{0}}\left(1+\tfrac{q_{1}^{2}}{q_{2}^{2}}\right)\left(1-\tfrac{q_{0}^{2}}{q_{1}^{2}}\right)J_{r}(q_{0}(x+y+T))\\
			& +\tfrac{1}{8q_{0}}\left(1-\tfrac{q_{1}^{2}}{q_{2}^{2}}\right)\left(1-\tfrac{q_{0}}{q_{1}}\right)^{2}J_{r}(q_{0}(x+y+T)+2q_{1}T)\\
			& +\tfrac{1}{8q_{0}}\left(1-\tfrac{q_{1}^{2}}{q_{2}^{2}}\right)\left(1+\tfrac{q_{0}}{q_{1}}\right)^{2}J_{r}(q_{0}(x+y+T)-2q_{1}T)\\
			& =k_{00}(x,y).
		\end{align*}
		The remaining cases are proved in a similar fashion. 
	\end{proof}
	\begin{figure}
		\includegraphics[width=1\linewidth]{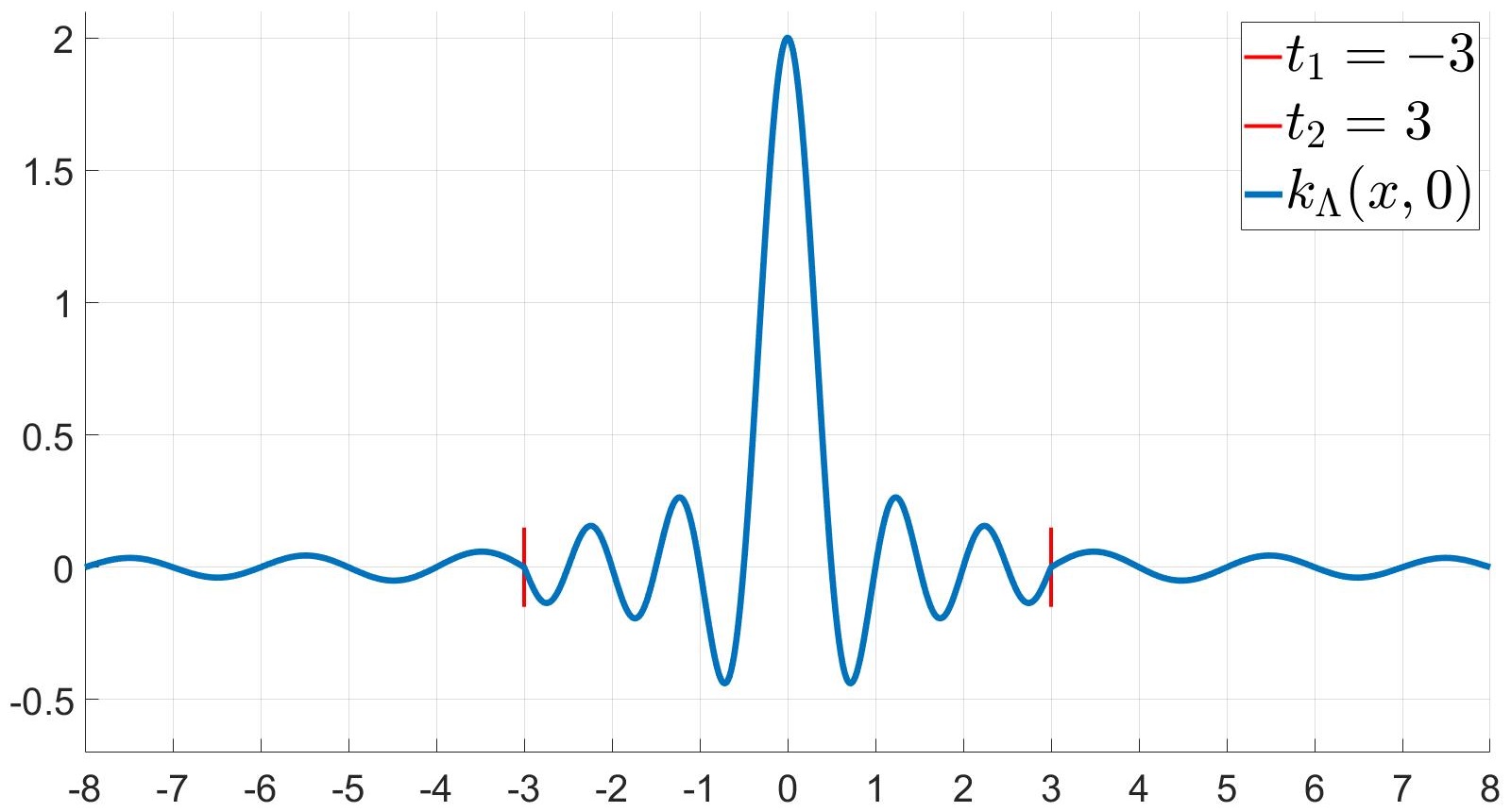} \caption{Graph of
			the reproducing kernel $k_{\Lambda}(0,\cdot)$ for the parameters $\Lambda=[0,\pi^{2}],p_{0}=1,p_{1}=\frac{1}{4},p_{2}=1$,
			$T=6$.} 
		\label{graphkerpar}
	\end{figure}
	
	A plot of the reproducing kernel is shown in Figure \ref{graphkerpar}.
	We also note the symmetry of the graph  respect to the
	$y$-axis. Note
	that the local bandwidth on the interval $[-3,3]$ is $1/\sqrt{p_1} =
	2$ which amounts to two zeros (two oscillations) per unit interval,
	whereas for $|x| \geq 3$ the local bandwidth is $1/\sqrt{p_0} =
	1/\sqrt{p_2} =1$, which amounts to one zero per unit interval. The
	kernel $k_\Lambda (0,\cdot )$ for $\pwo $  is the precise analogue of the
	$\sinc $-kernel $\tfrac{\sin x}{x}$ for the classical  Paley-Wiener
	space. 

	\section{A Density Theorem for Functions of Variable Bandwidth}
	
	\label{sec:dens} In this section, we derive necessary density conditions
	for sampling and interpolation in $PW_{\Lambda}(A_{p})$ with piecewise
	constant $p$ using a general theory about sampling in reproducing
	kernel Hilbert spaces from  \cite{fuehr}.
	
	\subsection{Sampling in reproducing kernel Hilbert spaces}

	We first summarize the
	definition of density and the necessary density conditions in
	reproducing kernel Hilbert spaces. For complete details see~\cite{fuehr}.

	For a Borel measure $\mu$ we define the the upper and lower Beurling
	densities of a separable set $X$ as the quantities 
	\begin{align*}
		D_{\mu}^{+}(X) & =\limsup_{r\to\infty}\sup_{x\in\bR}\dfrac{\#(X\cap B_{r}(x))}{\mu(B_{r}(x))},\\
		D_{\mu}^{-}(X) & =\liminf_{r\to\infty}\inf_{x\in\bR}\dfrac{\#(X\cap B_{r}(x))}{\mu(B_{r}(x))}
	\end{align*}

	The relevant measure for sampling in a reproducing kernel Hilbert
	space with kernel $k$ 
	is $d\mu\left(x\right)=k\left(x,x\right)dx$. 
	The dimension-free
	Beurling densities $D_{0}^{\pm}(X)$ are defined with respect to the
	measure $d\mu\left(x\right)=k\left(x,x\right)dx$:
	\[
	D_{0}^{+}(X)=\limsup_{r\to\infty}\sup_{x\in\bR}\dfrac{\#(X\cap B_{r}(x))}{\int_{B_{r}(x)}k(x,x)dx},\quad D_{0}^{-}(X)=\liminf_{r\to\infty}\inf_{x\in\bR}\dfrac{\#(X\cap B_{r}(x))}{\int_{B_{r}(x)}k(x,x)dx}
	\]

	To derive the desired density theorems in $PW_{\Lambda}(A_{p})$ with
	piecewise constant $p$, we will use the following special case of
	\cite[Thm.~2.2]{fuehr}. Note the crucial role of the reproducing
	kernel in the assumptions.  
	\begin{thm}
		\label{fuehrthm} %
		{} Assume $k$ is the reproducing kernel for a subspace $\mathcal{H}\subset L^{2}\left(\bR,dx\right)$
		and satisfies the following conditions:
		
		(i) \label{dee}Boundedness of diagonal: There exist constants
		$C_{1},C_{2}>0$ such that 
		\[
		C_{1}\leq k(x,x)\leq C_{2}
		\]
		for all $x\in\bR.$ 
		
		(ii) \label{ee} Weak localization property: For every $\epsilon>0$,
		there exists $r(\epsilon)>0$ such that 
		\[
		\sup_{x\in\bR}\int_{\bR\setminus B_{r(\epsilon)}(x)}|k(x,y)|^{2}\,dy<\epsilon^{2}.
		\]
		
		(iii) \label{ef} Homogeneous approximation property: Assume that a
		subset  $X\subseteq\bR$ satisfies a Bessel inequality 
		\begin{align*}
			\sum_{x\in X}|f(x)|^{2}\leq C\norm{f}_{\mathcal{H}}^{2}
		\end{align*}
		for all $f\in\mathcal{H}$. Then for every $\epsilon>0$ there exists
		$r(\epsilon)>0$ such that 
		\[
		\sup_{y\in\bR}\sum_{\underset{|x-y|>r(\epsilon)}{x\in X}}|k(x,y)|^{2}<\epsilon^{2}.
		\]
		
		Under these assumptions the following density conditions hold in $\cH
		$.
		
		(A) If $X$ is a set of stable sampling for $\mathcal{H}$, then $D_{0}^{-}(X)\geq1.$
		
		(B) If $X$ is a set of interpolation for $\mathcal{H}$, then $D_{0}^{+}(X)\leq1.$
	\end{thm}
	
	Instead of using  the measure $k(x,x)\, dx $, this result can be written in terms of any measure $\mu$
	equivalent\footnote{in the sense that $d\mu = h\,dx$ for some measurable function $h$ with $0 < c \leq h(x) < C$ for all $x\in \bR$ and some constants $c$ and $C$.} to the Lebesgue measure. For this define the upper and
	lower averaged traces of $k$ with respect to $\mu$ as 
	\begin{align*}
		\mathrm{tr_{\mu}^{+}}= & \limsup_{r\to\infty}\sup_{x\in\bR}\dfrac{1}{\mu(B_{r}(x))}\int_{B_{r}(x)}k(y,y)\,dy,\\
		\mathrm{tr_{\mu}^{-}}= & \liminf_{r\to\infty}\inf_{x\in\bR}\dfrac{1}{\mu(B_{r}(x))}\int_{B_{r}(x)}k(y,y)\,dy.
	\end{align*}
	Then the following reformulation of Theorem~\ref{fuehrthm} holds~\cite[Lemma~6.6]{apde}. 
	\begin{lem}
		\label{lem:changemeas} Let $d\mu=h\,dx$ be equivalent to the Lebesgue
		measure. Then $D_{0}^{-}(X)\geq1$
		holds, if and only if $D_{\mu}^{-}\left(X\right)\geq\mathrm{tr_{\mu}^{-}}$.
		Similarly $D_{0}^{+}\left(X\right)\leq1$, if and only if $D_{\mu}^{+}\left(X\right)\leq\mathrm{tr_{\mu}^{+}}$.
	\end{lem}
	
	
	\begin{proof}
		The inequality $D_{0}^{-}(X)\geq 1$ means that for all $\varepsilon>0$
		there is an $r_{\epsilon}>0$ such that for all $x\in \bR $ and all  $r>r_{\varepsilon}$
		\begin{equation}
			\#\left(X\cap B_{r}\left(x\right)\right)\geq\left(1-\varepsilon\right)\int_{B_{r}\left(x\right)}k\left(y,y\right)dy\,,\label{eq:denseps}
		\end{equation}
		or equivalently, 
		\[
		\frac{\#\left(X\cap B_{r}\left(x\right)\right)}{\mu\left(B_{r}\left(x\right)\right)}\geq\left(1-\varepsilon\right)\frac{1}{\mu\left(B_{r}\left(x\right)\right)}\int_{B_{r}\left(x\right)}k\left(y,y\right)dy\,.
		\]
		Written in terms of the Beurling density, this is 
		\[
		D_{\mu}^{-}\left(X\right)\geq\liminf_{r\to\infty}\inf_{x\in\rd}\frac{\int_{B_{r}\left(x\right)}k\left(y,y\right)dy}{\mu(B_{r}(x))}=\mathrm{tr}_{\mu}^{-}(k)\,.
		\]
		The converse is obtained by reading the argument backwards. 
	\end{proof}
	
	\subsection{Necessary density conditions in $\pwo $}
	\label{sec:...}
	For the Paley-Wiener space $\pwo $ we define  the positive measure
	$\mu = \mu_{p}$ by  
	\begin{align*}
		\mu (I) = \mu_{p}(I)=\int_{I}\dfrac{dx}{\sqrt{p(x)}}\,.
	\end{align*}
	Since $p $ is positive and bounded away from zero, $\mu_{p}$ is equivalent
	to Lebesgue measure. 
	
	In the case $\mu=\mu_{p}$ we write $D_{p}^{\pm},\trace_{p}^{\pm}$ instead
	of $D_{\mu_{p}}^{\pm}\left(X\right)$ etc. Our aim is to prove the
	following theorem for the spectral set $\Lambda$ being an interval.  
	\begin{thm}[Density theorem in $PW_{[0,\Omega]}(A_{p})$]
		\label{thm:density}Let $\Lambda = [0,\Omega ]$ 
		and $p$ a piecewise constant function. 
		\begin{enumerate}
			\item If $X$ is a set of stable sampling for $PW_{\Lambda}(A_{p})$, then
			$D_{p}^{-}(X)\geq\frac{|\Lambda^{1/2}|}{\pi}.$ 
			\item If $X$ is a set of interpolation for $PW_{\Lambda}(A_{p})$, then
			$D_{p}^{+}(X)\leq\frac{|\Lambda^{1/2}|}{\pi}.$ 
		\end{enumerate}
	\end{thm}
	
	The result is similar to \cite[Thms.~6.2, 6.3]{cpam} but with notable
	differences in the proofs. In particular, since $p$ is not smooth, not
	even continuous, we cannot use a Liouville transform to change $A_{p}$ 
	into an equivalent  Schr\"odinger operator and apply scattering theory
	(as was done in ~\cite{cpam}).  
	
	The plan of the proof is to verify the conditions of Theorem \ref{fuehrthm}
	for the reproducing kernel of $PW(A_{p})$.
	We will work our way through a series of Lemmas that verify the conditions
	of Theorem~\ref{fuehrthm} and then 
	obtain  the value for the averaged trace of the kernel. 

	\subsection{Fine properties of the reproducing kernel of $\pwo $}
	We first verify that the properties of Theorem~\ref{fuehrthm} are
	satisfied for the reproducing kernel of $\pwo $. 
	
	\begin{lem}[Boundedness of the diagonal]
		\label{bdddiag} Let $\Lambda\subset\bR_{0}^{+}$ be  a set of finite measure,
		$p$ a piecewise constant function and $k_{\Lambda}$ the reproducing
		kernel for $PW_{\Lambda}(A_{p})$. Then there exist constants $C_{1},C_{2}>0$
		such that 
		\[
		C_{1}\leq k_{\Lambda}(x,x)\leq C_{2}
		\]
		for all $x\in\bR$. 
	\end{lem}
	
	\begin{proof}
		The uniform boundedness of $k_{\Lambda}$ (hence the existence of
		$C_{2}$) follows from the uniform boundedness of the fundamental
		solutions of $A_p$ by Lemma
		\ref{unifbdd} and the formula for the  reproducing kernel $k_{\Lambda}$ in Theorem
		\ref{thmunifbdd}. On the other hand, 
		\begin{align*}
			k_{\Lambda}(x,x) & =\dfrac{1}{2\pi}\int_{\Lambda^{1/2}}\dfrac{\frac{1}{q_{0}}|\Phi^{+}(u^{2},x)|^{2}+\frac{1}{q_{n}}|\Phi^{-}(u^{2},x)|^{2}}{\kappa(u)}\,du\\
			& \geq\dfrac{1}{2\pi q_{0}}\int_{\Lambda^{1/2}}\dfrac{|\Phi^{+}(u^{2},x)|^{2}}{\kappa(u)}\,du\\
			& \geq\min_{0\leq j\leq n}\inf_{x\in I_{j}}\dfrac{1}{2\pi q_{0}}\int_{\Lambda^{1/2}}\dfrac{|a_{j}^{+}(u^{2})e^{iq_{j}xu}+b_{j}^{+}(u^{2})e^{-iq_{j}xu}|^{2}}{\kappa(u)}\,du\\
			& \geq\min_{0\leq j\leq n}\dfrac{1}{2\pi q_{0}}\int_{\Lambda^{1/2}}\dfrac{\left||a_{j}^{+}(u^{2})|-|b_{j}^{+}(u^{2})|\right|^{2}}{\kappa(u)}\,du=C_{1}.
		\end{align*}
		Since $|a_{j}^{+}(u^{2})|-|b_{j}^{+}(u^{2})|=\frac{q_{j}}{q_{n}}$
		by \eqref{id1} and  $\kappa$ is bounded, 
		we conclude that $C_{1}>0$. 
	\end{proof}
	Next we show that the kernel $k_{\Lambda}$ exhibits off-diagonal
	decay. Here we need that $\Lambda$ is an interval. 
	\begin{lem}
		\label{kerestlem} Let $\Lambda=[0,\Omega]$ for some $\Omega>0$,
		$p$ a piecewise constant function and $k_{\Lambda}$ the reproducing
		kernel for $PW_{[0,\Omega]}(A_{p})$. Then for some constant $C>0$
		\begin{align}
			|k_{\Lambda}(x,y)|\leq\dfrac{C}{1+|x-y|}\qquad \text{ for all } x,y\in
			\bR \, .\label{kineq}
		\end{align}
	\end{lem}
	
	\begin{proof}
	Let $p$ be a piecewise constant function for some $n\in\bN$. Recall
	the expression \eqref{eq:kerexpo} for the reproducing kernel
	\begin{align*}
		k_{\Lambda}(x,y)=\sum_{k=1}^{m(x,y)}\alpha_{k}(x,y)J(\beta_{k}(x,y)),\quad m(x,y)\in\bN,\alpha_{k}(x,y),\beta_{k}(x,y)\in\bR,1\leq k\leq m(x,y)
	\end{align*}
	with $J(s)=\frac{1}{2\pi}\int_0^{\Omega ^{1/2}}\frac{e^{isu}}{\kappa(u)}\,du$
	and $\beta_{k}(x,y)=c_{k}^{(jl)}\pm q_{j}x\pm q_{l}y$ for
	$x\in I_j,y\in I_l$. 

        In a first step we  show that $J(s)=\mathcal{O}(s^{-1})$ for $\left|s\right|$ large enough.
	Indeed, as $\kappa$
	is (infinitely) differentiable on $(0,\infty)$ with bounded derivatives
	and $0<\frac{1}{\kappa(u)}\leq q_{0}q_{n}$ for all $u\in(0,\infty)$,
	integration by parts yields 
	\begin{align}
		J(s) & =\dfrac{1}{2\pi is}\left(\dfrac{e^{is\Omega^{1/2}}}{\kappa(\Omega^{1/2})}-\dfrac{1}{\kappa(0^{+})}\right)+\dfrac{1}{2\pi is}\int_{0}^{\Omega^{1/2}}\frac{\kappa'(u)}{\kappa(u)^{2}}e^{isu}\label{jbp}\\
		|J(s)| & \leq\dfrac{1}{2\pi|s|}\left\{ \dfrac{1}{|\kappa(\Omega^{1/2})|}+\dfrac{1}{|\kappa(0^{+})|}+\int_{0}^{\Omega^{1/2}}\dfrac{\left|\kappa'(u)\right|}{\kappa(u)^{2}}\,du\right\} ,\nonumber 
	\end{align}
	where $\kappa(0^{+})=\lim_{u\downarrow0}\kappa(u)$. Therefore for
	all $s\neq0$, 
	\begin{align*}
		|J(s)|\leq\frac{q_{0}q_{n}}{2\pi|s|}\left(2+q_{0}q_{n}\Omega^{1/2}\sup
		_{0\leq u \leq \Omega ^{1/2}} |\kappa'(u)|\right).
	\end{align*}
	%
In the second step, we  verify that for all $k$ there exist for some
$N_{k},r_{k}>0$, such that 
	\begin{equation}
			|\beta_{k}(x,y)|\geq{N_{k}}(1+|x-y|)\label{eq:bk}
	\end{equation}
	 for all $x,y\in\bR$ satisfying $|x-y|>r_{k}$. We proceed as follows: let $a>0$ such that the knots $\{t_{k}\}_{k=1}^{n}$
	of $p$ are contained in $[-a,a]$. Consider cases where $x$ and
	$y$ take values on the intervals 
	\[
	(-\infty,-a)\subset I_{0},\quad[-a,a]\supset I_{1}\cup\ldots I_{n-1},\quad(a,\infty)\subset I_{n}
	\]
	with $|x-y|$ large. 
	
	\begin{itemize}[leftmargin=*]
		 
	\item Suppose $x$ and $y$ belong to the same unbounded interval, say without
	loss of generality $x,y<-a$ . Then 
	$
	\beta_{k}(x,y)=c_{k}\pm q_{0}(x\pm y)
	$.
	As  $|x\pm y|\geq|x-y|$ we obtain \eqref{eq:bk}
	for $|x-y|$ sufficiently large, i.e., $|x-y|\geq r_{k}$. 
	
	\item Likewise the case $|x|\geq a$ and $|y|\leq a$ implies
          \eqref{eq:bk}. 
	
	\item The remaining case is $x<-a$
	and $y>a$, or by symmetry $x>a$ and $y<-a$. Then
        $|q_{0}x+q_{n}y|$ may be bounded, although  $|x-y| \to
        \infty$, and (\ref{eq:bk}) is violated. Here we use the 
        original formulation of $k_{\Lambda}$.  We have  
	\begin{align*}
		\vartheta(u,x,y) & =\frac{1}{q_{0}}\overline{\Phi^{+}(u^{2},x)}\Phi^{+}(u^{2},y)
            +\frac{1}{q_{n}}\overline{\Phi^{-}(u^{2},x)}\Phi^{-}(u^{2},y)\\ 
          & =\frac{1}{q_0} \Big( \overline{a_0^+(u^2)}  e^{-iq_0x}+\overline{b_0^+(u^2)}  e^{iq_0x}\Big)
e^{iq_nyu}+\frac{1}{q_n} e^{iq_0 xu} \Big( a_n^-(u^2)  e^{iq_nyu}+b_n^-(u^2) e^{-iq_nyu}  \Big)\\
    &=\Big( \frac{\overline{b_0^+(u^2)}}{q_0} + \frac{a_n^-(u^2)}{q_n}
      \Big) e^{i(q_0x+q_ny)u} +
      \frac{\overline{a_{0}^{+}(u^{2})}}{q_{0}}e^{-i(q_{0}x-q_{n}y)u}+\frac{a_{0}^{+}(u^{2})}{q_{0}}e^{i(q_{0}x-q_{n}y)u}              \\
    & =\frac{\overline{a_{0}^{+}(u^{2})}}{q_{0}}e^{-i(q_{0}x-q_{n}y)u}+\frac{a_{0}^{+}(u^{2})}{q_{0}}e^{i(q_{0}x-q_{n}y)u},
	\end{align*}
        where we have used \eqref{wron1} and \eqref{wron3} of Lemma~\ref{wronskthm}.
	This implies that there are no exponentials containing $\pm(q_0x +
        q_ny)$. Meanwhile, observe that
        $$
        |q_{0}x-q_{n}y|=q_{n}y -q_{0}x \geq \min (q_0,q_n)  (y-x)
        $$
   and      \eqref{eq:bk} follows.
\end{itemize}
To summarize, the decay estimate on $J$ and \eqref{eq:bk} imply that
there exists a  $C>0$ such that 
\[
|k_{\Lambda}(x,y)|\leq\frac{C}{1+|x-y|}
\]
for all $x,y\in \bR$. 
\end{proof}

	\begin{lem}
		[Weak localization] \label{WLemma} Let $\Lambda=[0,\Omega]$ for
		some $\Omega>0$, $p$ be a piecewise constant function and $k_{\Lambda}$
		the reproducing kernel for $PW_{[0,\Omega]}(A_{p})$. Then for every
		$\epsilon>0$, there exists $r(\epsilon)>0$ such that 
		\begin{equation}
			\sup_{x\in\bR}\int_{|x-y|>r(\epsilon)}|k_{\Lambda}(x,y)|^{2}\,dy<\epsilon^{2}.\label{WL}
		\end{equation}
	\end{lem}
	
	\begin{proof}
		By Lemma \ref{kerestlem}, we can find $C>0$ such that for all
		$x,y\in\bR$ 
		\begin{align*}
			|k_{\Lambda}(x,y)|\leq\dfrac{C}{1+|x-y|}.
		\end{align*}
		Let $\epsilon>0$ and take $r(\epsilon)>\max\{b,\frac{4C^{2}}{\epsilon^{2}}-1\}$.
		Then for a fixed $x\in\bR$, 
		\begin{align*}
			\int_{|x-y|>r(\epsilon)}|k_{\Lambda}(x,y)|^{2}\,dy & \leq\int_{|x-y|>r(\epsilon)}\dfrac{C^{2}}{(1+|x-y|)^{2}}\,dy=\int_{r(\epsilon)}^{\infty}\dfrac{2C^{2}}{(1+z)^{2}}\,dz\\
			& =\frac{2C^{2}}{1+r(\epsilon)}<\frac{\epsilon^{2}}{2}.
		\end{align*}
		Taking the supremum over all $x\in\bR$ proves \eqref{WL}. 
	\end{proof}
	For the proof of the homogeneous approximation property, we recall
	that a set $X\subset\bR$ is 
        \emph{relatively separated} if 
	\begin{align}
		\text{rel}(X)=\max_{x\in\bR}\#(X\cap[x,x+1])\label{rel}
	\end{align}
	is  finite. Lemma 3.7 in \cite{fuehr} implies that if $\{k(x,\cdot):x\in X\}$
	is a Bessel sequence for $PW_{\Lambda}\left(\bR\right)$ then $X$
	is relatively separated. %

	\begin{lem}
		[Homogeneous approximation property] \label{HAP} Let $\Lambda=[0,\Omega]$
		for some $\Omega>0$, $p$ a piecewise constant function and $k_{\Lambda}$
		the reproducing kernel for $PW_{[0,\Omega]}(A_{p})$. Suppose $X\subset\bR$
		such that $\{k_{\Lambda}(x,\cdot):x\in X\}$ is a Bessel sequence
		in $PW_{[0,\Omega]}(A_{p})$. Then for every $\epsilon>0$, there
		exists $r(\epsilon)>0$ such that 
		\begin{equation}
			\sup_{y\in\bR}\sum_{\underset{|x-y|>r(\epsilon)}{x\in X}}|k_{\Lambda}(x,y)|^{2}<\epsilon^{2}.\label{HAPconc}
		\end{equation}
	\end{lem}
	
	\begin{proof}
		Fix $y\in\bR$. By Lemma \ref{kerestlem}, there exists $r,C>0$
		such that 
		\begin{align*}
			\sum_{\underset{|x-y|>r}{x\in X}}|k_{\Lambda}(x,y)|^{2}\leq\sum_{\underset{|x-y|>r}{x\in X}}\dfrac{C^{2}}{1+|x-y|^{2}}.
		\end{align*}
		As $X$ is relatively separated, the right hand side
                can be made  smaller than a given $\epsilon>0$ for $r$ large enough.%
	\end{proof}

	Lemmas~\ref{WLemma} and \ref{HAP} hold for arbitrary spectral sets $\Lambda $
	of finite measure, although the off-diagonal decay of $k_\Lambda $ in
	\eqref{kineq}  no
	longer holds.  
	See~\cite[Lems.~C.1.2, C.2.3]{CelizDissertation} for the treatment of the general case. 

	{} %

	\subsection{The averaged trace of the reproducing kernel}
	Next we compute the averaged trace of $k$ required for the abstract
	density theorem (Lemma~\ref{lem:changemeas}). For these arguments
	$\Lambda $ may be an arbitrary set of finite measure. 
	
	To rewrite the diagonal of the reproducing kernel, we 
	define  the auxiliary functions 
	\begin{align*}
		h_{j}^{(1)}(u) & =\frac{|a_{j}^{+}(u^{2})|^{2}}{q_{0}}+\frac{|a_{j}^{-}(u^{2})|^{2}}{q_{n}}\\
		h_{j}^{(2)}(u) & =\frac{\overline{a_{j}^{+}(u^{2})}b_{j}^{+}(u^{2})}{q_{0}}+\frac{\overline{a_{j}^{-}(u^{2})}b_{j}^{-}(u^{2})}{q_{n}}\,
	\end{align*}
	for $0\leq j\leq n$
	and $u\in(0,\infty)$.
	\begin{lem}
		\label{lemkxx} Let $\Lambda\subset\bR_{0}^{+}$ be of finite measure,
		$p$ a step function with $n$ jumps 
		and $k_{\Lambda}$ the reproducing kernel for $PW_{\Lambda}(A_{p})$.
		Then 
		\begin{align}
			k_{\Lambda}(y,y)= \sum _{j=0}^n
			\Big(\dfrac{1}{\pi}\int_{\Lambda^{1/2}}\dfrac{h_{j}^{(1)}(u)}{\kappa(u)}\,du+\dfrac{1}{\pi}\emph{Re}\mathcal{F}\big(\dfrac{h_{j}^{(2)}}{\kappa}\cdot\chi_{\Lambda^{1/2}}\big)(2q_{j}y)
			\Big) \chi _{(t_j,t_{j+1}]}(y) \,.\label{kxxthm}
		\end{align}
	\end{lem}
	
	\begin{proof}
		We evaluate the diagonal of $k_{\Lambda}$ using Proposition \ref{thmunifbdd}.
		Fix $y\in I_{j}$ for some $0\leq j\leq n$. Using the formula \eqref{kref}
		for $k_{\Lambda}$, we obtain 
		\begin{align*}
			\vartheta(u,y,y) & =\frac{1}{q_{0}}|\Phi^{+}(u^{2},y)|^{2}+\frac{1}{q_{n}}|\Phi^{-}(u^{2},y)|^{2}\\
			& =\frac{1}{q_{0}}|a_{j}^{+}(u^{2})e^{iuq_{j}y}+b_{j}^{+}(u^{2})e^{-iuq_{j}y}|^{2}+\frac{1}{q_{n}}|a_{j}^{-}(u^{2})e^{iuq_{j}y}+b_{j}^{-}(u^{2})e^{-iuq_{j}y}|^{2}\\
			& =\left(\frac{|a_{j}^{+}(u^{2})|^{2}}{q_{0}}+\frac{|b_{j}^{+}(u^{2})|^{2}}{q_{0}}\right)+\left(\frac{|a_{j}^{-}(u^{2})|^{2}}{q_{n}}+\frac{|b_{j}^{-}(u^{2})|^{2}}{q_{n}}\right)\\
			& +2\textrm{Re}\left[\left(\frac{\overline{a_{j}^{+}(u^{2})}b_{j}^{+}(u^{2})}{q_{0}}+\frac{\overline{a_{j}^{-}(u^{2})}b_{j}^{-}(u^{2})}{q_{n}}\right)e^{-2iuq_{j}y}\right]\\
			& =2\left(\frac{|a_{j}^{+}(u^{2})|^{2}}{q_{0}}+\frac{|a_{j}^{-}(u^{2})|^{2}}{q_{n}}\right)+2\textrm{Re}\left[\left(\frac{\overline{a_{j}^{+}(u^{2})}b_{j}^{+}(u^{2})}{q_{0}}+\frac{\overline{a_{j}^{-}(u^{2})}b_{j}^{-}(u^{2})}{q_{n}}\right)e^{-2iuq_{j}y}\right]\,,
		\end{align*}
		where in the last identity we have used  \eqref{hidentity}.
		Then 
		\begin{align*}
			k_{\Lambda}(y,y) & =\dfrac{1}{2\pi}\int_{\Lambda^{1/2}}\dfrac{\vartheta(u,y,y)}{\kappa(u)}\,du\\
			& =\dfrac{1}{\pi}\int_{\Lambda^{1/2}}\dfrac{h_{j}^{(1)}(u)}{\kappa(u)}\,du+\dfrac{1}{\pi}\int_{\Lambda^{1/2}}\dfrac{\textrm{Re}(h_{j}^{(2)}(u)e^{-2iq_{j}yu})}{\kappa(u)}\,du\\
			& =\dfrac{1}{\pi}\int_{\Lambda^{1/2}}\dfrac{h_{j}^{(1)}(u)}{\kappa(u)}\,du+\dfrac{1}{\pi}\textrm{Re}\,\mathcal{F}\left(\dfrac{h_{j}^{(2)}}{\kappa}\cdot\chi_{\Lambda^{1/2}}\right)(2q_{j}y)
		\end{align*}
		as claimed. 
	\end{proof}
	Note that the first term $\int h^{(1)}_j$ depends only on the
	interval $I_{j}$ , but not on $y\in I_j$.
	
	In the following theorem we  
	compute the averaged trace of the kernel $k$. We recall that  the
	notation $f\lesssim g$ means that  there exists
	$C\geq0$ such that $f\leq Cg$. 
	\begin{thm}
		\label{limmup} Let $\Lambda\subset\bR_{0}^{+}$ be of finite measure,
		$p$ a piecewise constant function and $k_{\Lambda}$ the reproducing
		kernel of $PW_{\Lambda}(A_{p})$. Suppose $I$ is a large interval
		such that $I\not\subset[t_{1},t_{n}]$. Then 
		\begin{align*}
			\left|\dfrac{1}{\mu_{p}(I)}\int_{I}k_{\Lambda}(y,y)\,dy-\dfrac{|\Lambda^{1/2}|}{\pi}\right|\lesssim\mu_{p}(I)^{-1/2}\,.
		\end{align*}
		In particular, 
		\[
		\emph{tr}^{-}=
		\liminf_{r\to\infty}\inf_{x\in\bR}\dfrac{1}{\mu_{p}(B_{r}(x))}\int_{B_{r}(x)}k_{\Lambda}(y,y)\,dy =
		\frac{|\Lambda^{1/2}|}{\pi} \, , 
		\]
		and likewise $\emph{tr}^{+}= \frac{|\Lambda^{1/2}|}{\pi}$.
	\end{thm}

	\begin{proof}
		Let $p$ be a piecewise constant function with $n$ jumps at $\{t_{k}\}_{k=1}^{n}$. We may assume that
		the interval $I=[\alpha,\beta]\supseteq[t_{1},t_{n}]$. The two other
		cases are easier to handle. Since $ \frac{1}{\sqrt{p(x)}}
		\chi _{(t_j,t_{j+1}]}(x) =  q_j$ for $x\in (t_j, t_{j+1}]$, we obtain 
		\begin{align}
			\mu_{p}(I) & \int _{\alpha } ^\beta \frac{1}{\sqrt{p(x)}} \, dx
			=\left[q_{0}(t_{1}-\alpha)+q_{n}(\beta-t_{n})\right]+\sum_{j=1}^{n-1}q_{j}(t_{j+1}-t_{j})
			\, ,\label{measequiv}
		\end{align}
		and $\mu _p(I) \geq  (\beta -\alpha ) \min _{j} q_j$. 
		We first observe that 
		\[
		\dfrac{1}{\mu_{p}(I)}\int_{I}k_{\Lambda}(y,y)\,dy=\dfrac{1}{\mu_{p}(I)}\int_{I\setminus[t_{1},t_{n}]}k_{\Lambda}(y,y)\,dy+\dfrac{1}{\mu_{p}(I)}\int_{[t_{1},t_{n}]}k_{\Lambda}(y,y)\,dy\,.
		\]
		Since the second term in this some is of the form $C\mu_{p}(I)^{-1}$
		we can focus on the first term. We use the representation of the diagonal
		of $k_{\Lambda}$ obtained in Lemma \ref{lemkxx} and consider first
		the expression
		\[
		T_{1}=\dfrac{1}{\mu_{p}(I)\pi}\Big(\int_{\alpha{}_{\alpha}}^{t_{1}}\int_{\Lambda^{1/2}}\dfrac{h_{0}^{(1)}(u)}{\kappa(u)}\,du\,dy+\int_{t_{n}}^{\beta}\int_{\Lambda^{1/2}}\dfrac{h_{0}^{(1)}(u)}{\kappa(u)}\,du\,dy\Big)\,.
		\]
		Using  the initial conditions $a_{0}^{-}=b_{n}^{+}=0$, $a_{n}^{+}=b_{0}^{-}=1$
		and the  definition of $\kappa$ in \eqref{specmeasdmu}, we obtain 
		\[
		h_{0}^{(1)}(u)=\frac{|a_{0}^{+}(u^{2})|^{2}}{q_{0}}=q_{0}\kappa(u),\quad h_{n}^{(1)}(u)=\frac{|b_{n}^{-}(u^{2})|^{2}}{q_{n}}=q_{n}\kappa(u)\,,
		\]
		and consequently 
		\begin{align*}
			T_{1} &=\dfrac{|\Lambda^{1/2}|}{\mu_{p}(I)\pi}\left(q_{0}(t_{1}-\alpha)+q_{n}(\beta-t_{n})\right)\\
			&= \dfrac{|\Lambda^{1/2}|}{\mu_{p}(I)\pi} \Big( \mu _p(I) -
			\sum_{j=1}^{n-1}q_{j}(t_{j+1}-t_{j})\Big) =
			\dfrac{|\Lambda^{1/2}|}{\pi} + \cO \big( \dfrac{1}{\mu _p(I)} \big).
		\end{align*}
		This leads to the estimate 
		\begin{align}
			\left|T_{1}-\dfrac{|\Lambda^{1/2}|}{\pi}\right|\lesssim\mu_{p}(I)^{-1}\lesssim\mu_{p}(I)^{-1/2}.\label{newT1til}
		\end{align}
		For the second term, we make a similar computation. Let 
		\begin{align*}
			T_{2} & =\dfrac{1}{\mu_{p}(I)\pi}\int_{I\setminus[t_{1},t_{n}]}\textrm{Re}\mathcal{F}\Big(\dfrac{h_{j}^{(2)}}{\kappa}\cdot\chi_{\Lambda^{1/2}}\Big)(2q_{j}y)\,dy\\
			& =\dfrac{1}{\mu_{p}(I)\pi}\textrm{Re}\left(\int_{\alpha}^{t_{1}}\int_{\Lambda^{1/2}}\dfrac{h_{0}^{(2)}(u)e^{-2iq_{0}yu}}{\kappa(u)}\,du\,dy+\int_{t_{n}}^{\beta}\int_{\Lambda^{1/2}}\dfrac{h_{n}^{(2)}(u)e^{-2iq_{n}yu}}{\kappa(u)}\,du\,dy\right)\,.
		\end{align*}
		Interchanging the order of integration and evaluating  the
		integral over $y$ first, we have 
		\begin{align*}
			T_{2} & =\dfrac{1}{\mu_{p}(I)\pi}\textrm{Re}\int_{\Lambda^{1/2}}\dfrac{h_{0}^{(2)}(u)e^{-2iq_{0}(t_{1}+\alpha)u}}{\kappa(u)}\cdot\dfrac{\sin(q_{0}(t_{1}-\alpha)u)}{q_{0}u}\,du\\
			& +\dfrac{1}{\mu_{p}(I)\pi}\textrm{Re}\int_{\Lambda^{1/2}}\dfrac{h_{n}^{(2)}(u)e^{-2iq_{n}(\beta+t_{n})u}}{\kappa(u)}\cdot\dfrac{\sin(q_{n}(\beta-t_{n})u)}{q_{n}u}\,du.
		\end{align*}
		With  the  Cauchy-Schwartz inequality,  this can be estimated as 
		\begin{align*}
			\left|T_{2}\right| &  \leq\dfrac{1}{\mu_{p}(I)\pi}\left(C_0\left(\int_{\Lambda^{1/2}}
			\dfrac{\sin^{2}(q_{0}(t_{1}-\alpha)u)}{(q_{0}u)^2}du\right)^{1/2}+C_n\left(\int_{\Lambda^{1/2}}\dfrac{\sin^{2}(q_{n}(\beta-t_{n})u)}{(q_{n}u)^2}du\right)^{1/2}\right)\\ 
			&
			=\dfrac{C_{0}}{\mu_{p}(I)\pi}\left(\frac{t_{1}-\alpha}{q_{0}}\int_{q_{0}(t_{1}-\alpha)\Lambda^{1/2}}\!\sinc
			^{2}\omega\,d\omega\right)^{1/2}+\dfrac{C_{n}}{\mu_{p}(I)\pi}\left(\frac{\beta-t_{n}}{q_{n}}\int_{q_{n}(\beta-t_{n})\Lambda^{1/2}}\!\sinc
			^{2}\omega\,d\omega\right)^{1/2}\\
			& \leq\dfrac{C}{\mu_{p}(I)\pi}\mu_{p}(I)^{1/2}\leq
			C'\mu_{p}(I)^{-1/2} \, .
		\end{align*}
		{} Putting these estimates together we obtain 
		\begin{align}
			\left|\dfrac{1}{\mu_{p}(I)}\int_{I}k_{\Lambda}(y,y)\,dy-\dfrac{|\Lambda^{1/2}|}{\pi}\right|=\left|
			T_1+T_2 - \dfrac{|\Lambda^{1/2}|}{\pi}\right| \leq C\mu_{p}(I)^{-1/2}\,.\label{dens}
		\end{align}
	\end{proof}
	
	\subsection{Proof of Theorem \ref{thm:density}}
	\begin{proof}
		Theorem \ref{thm:density} now follows from the general density theorem
		as stated in Theorem~\ref{fuehrthm}. In Lemmas~\ref{bdddiag} --- \ref{HAP} we
		have verified that the conditions on the reproducing kernel are
		satisfied, and in Theorem~\ref{limmup} we have computed the critical
		density given by the averaged trace of the kernel. 
	\end{proof}
\begin{rem*}
	We finish by mentioning that the assumptions on the spectral set
	$\Lambda $ can be weakened and Theorem~\ref{thm:density} remains true, when
	$\Lambda $ is just a set of finite measure.
	
	If the spectral set is not an interval $[0,\Omega ]$, then Proposition
	\ref{kerestlem} no longer holds, and therefore the proofs of Lemmas
	\ref{WLemma} and \ref{HAP} do  not work for general spectral
	sets. Nevertheless the conclusions of \ref{WLemma} and \ref{HAP} still
	hold whenever $\Lambda $ has finite measure. The proofs are more
	delicate and the interested reader is referred to~\cite{CelizDissertation}.
\end{rem*}

\def\cprime{$'$}


\end{document}